\documentclass[reqno]{amsart}

\usepackage[dvipsnames]{xcolor}

\usepackage{iftex}
\ifPDFTeX 
  \usepackage[utf8]{inputenc}
  \usepackage[T1]{fontenc}
\else 
  \usepackage{fontspec}
\fi
\usepackage{lmodern}

\usepackage{geometry}
\usepackage[textwidth=28mm]{todonotes}
\usepackage{mathtools}
\usepackage{yfonts}
\usepackage[foot]{amsaddr}

\usepackage[pdfencoding=auto]{hyperref}
\hypersetup{
    colorlinks,
    citecolor=LimeGreen,
    filecolor=black,
    linkcolor=blue,
    urlcolor=black
}

\newtheorem{Proposition}{Proposition}
\newtheorem{Theorem}{Theorem}
\newtheorem{Lemma}{Lemma}

\newcommand{\dt}{\,\partial_t\, }
\newcommand{\dx}{\,\partial_x\, }

\newcommand{\Dt}{\Delta t}
\newcommand{\Dx}{\Delta x}
\newcommand{\Dv}{\Delta v}

\newcommand{\dxx}{\,\partial_{xx}\, }

\newcommand{\dd}{\,\mathrm{d}}

\newcommand{\R}{\mathbb{R}}
\newcommand{\T}{\mathbb{T}}

\newcommand{\TT}{\mathsf{T}}

\newcommand{\II}{\mathcal{I}}
\newcommand{\J}{\mathcal{J}}

\newcommand{\ft}{\widetilde{f}}
\newcommand{\Ft}{{\widetilde{F}}}
\newcommand{\rhot}{\widetilde{\rho}}
\newcommand{\Jt}{\widetilde{J}}
\newcommand{\St}{\widetilde{S}}
\newcommand{\fluxFt}{\widetilde{\mathcal{F}}}
\newcommand{\fluxGt}{\widetilde{\mathcal{G}}}

\newcommand{\Ftbf}{\mathbf{\Ft}}
\newcommand{\Fbf}{\mathbf{F}}
\newcommand{\Gbf}{\mathbf{G}}
\newcommand{\Jbf}{\mathbf{J}}
\newcommand{\Sbf}{\mathbf{S}}

\newcommand{\Phibf}{\boldsymbol{\Phi}}
\newcommand{\rhobf}{\boldsymbol{\rho}}

\newcommand{\rhotbf}{\boldsymbol{\rhot}}
\newcommand{\Jtbf}{\mathbf{\Jt}}
\newcommand{\Stbf}{\mathbf{\St}}

\newcommand{\e}{\,\mathrm{e}}

\newcommand{\iph}{{i+\frac{1}{2}}}
\newcommand{\imh}{{i-\frac{1}{2}}}
\newcommand{\jph}{{j+\frac{1}{2}}}
\newcommand{\jmh}{{j-\frac{1}{2}}}

\newcommand{\inv}{^{-1}}

\newcommand{\rhoinfs}{\rho^{\infty,*}}

\newcommand{\tin}{{\mathrm{in}}}

\newcommand{\lla}{\left\langle}
\newcommand{\rra}{\right\rangle}
\newcommand{\lbr}{\left\lbrace}
\newcommand{\rbr}{\right\rbrace}

\newcommand{\overbar}[1]{\mkern 1.5mu\overline{\mkern-1.5mu#1\mkern-1.5mu}\mkern 1.5mu}

\numberwithin{equation}{section}

\title[Entropy-dissipating scheme for a nonlinear kinetic model]{Discrete $H$-theorem for a finite volume discretization of a nonlinear kinetic system: application to hypocoercivity}

\author{Marianne Bessemoulin-Chatard}
\address[Marianne Bessemoulin-Chatard]{Nantes Université, CNRS, Laboratoire de Mathématiques Jean Leray, LMJL,
UMR 6629, F-44000 Nantes, France}
\author{Tino Laidin}
\address[Tino Laidin]{Univ Brest, CNRS UMR 6205, Laboratoire de Mathématiques de Bretagne Atlantique, F-29200 Brest, France}
\author{Thomas Rey}
\address[Thomas Rey]{Université Côte d’Azur, CNRS, LJAD, Parc Valrose, F-06108 Nice, France}
\email{marianne.bessemoulin@univ-nantes.fr}
\email{tino.laidin@univ-brest.fr}
\email{thomas.rey@univ-cotedazur.fr}

\begin{document}

\begin{abstract}
    In this article, we study the long-time behavior of a finite-volume discretization for a nonlinear kinetic reaction model involving two interacting species. Building upon the seminal work of [Favre, Pirner, Schmeiser, ARMA, 2023], we extend the discrete exponential convergence to equilibrium result established in [Bessemoulin-Chatard, Laidin, Rey, IMAJNA, 2025], which was obtained in a perturbative framework using weighted $L^2$ estimates. The analysis applies to a broader class of exponentially decaying initial data, without requiring proximity to equilibrium, by exploiting the properties of the Boltzmann entropy. The proof relies on the propagation of the initial $L^\infty$ bounds, derived from monotonicity properties of the scheme, allowing controlled linearizations within the nonlinear entropy estimates. Moreover, we show that the time-discrete dissipation inherent to the numerical scheme plays a crucial stabilizing role, providing control over the nonlinear terms.
    
\end{abstract}



\maketitle

\section{Introduction}
This work investigates the long-time behavior of an implicit-in-time finite-volume discretization for a nonlinear kinetic reaction model describing two interacting species in one-dimensional position–velocity phase space. We focus on the following system introduced in~\cite{NeumannSchmeiser2016}
\begin{align}
    &\dt f_1 + v\dx f_1 = \chi_1 - \rho_2 f_1, \label{eq_1_nonlin}\\
    &\dt f_2 + v\dx f_2 = \chi_2 - \rho_1 f_2, \label{eq_2_nonlin}
\end{align}
where $f_1$ and $f_2$ denote the phase-space densities of two chemical reactants, A and B. These species are produced by the decomposition of a third substance C (whose density is assumed to be constant) with nonnegative velocity profiles $\chi_1$ and $\chi_2$, and can also recombine to form C, thus leaving the system. The functions $f_1$ and $f_2$ depend on time $t \ge 0$, position $x \in \T$ (the one-dimensional torus), and velocity $v \in \R$. The reaction probability depends on the spatial density of the reaction partner, defined by
\begin{equation}\label{def_densities}
    \rho_{k}(t,x)\coloneq\int_\R f_k(t,x,v)\,\dd v, \qquad k=1,\,2.
\end{equation}
The system \eqref{eq_1_nonlin}--\eqref{eq_2_nonlin} is complemented with the initial condition
\begin{equation}\label{CI}
    f_1(0,x,v)=f_1^\tin(x,v)\geq 0,\quad f_2(0,x,v)=f_2^\tin(x,v)\geq 0.
\end{equation}
Since the chemical reaction is assumed to be reversible, we impose the compatibility condition $\int_\R (\chi_1 -\chi_2)\,\dd v=0$. Moreover, we make the following standard assumptions on the velocity profiles $\chi_k$, which are given positive functions of $v$. For $k=1,\,2$, we require
\begin{equation}\label{hyp_chi}
    \begin{gathered}
        \int_\R \chi_k\,\dd v=1,\qquad \int_\R v\,\chi_k\,\dd v=0,\\
        D_k\coloneq\int_\R v^2\,\chi_k\,\dd v<\infty,\quad Q_k\coloneq\int_\R v^4\,\chi_k\,\dd v<\infty.
    \end{gathered}
\end{equation}
In \cite{FavrePirnerSchmeiser2023}, additional linear relaxation terms of the form $\rho_k \chi_k - f_k$, $k=1,\,2$ were introduced to model thermalization effects relevant in semiconductor physics. While such relaxation terms could also be included within our analysis, the present work focuses on the nonlinear reaction terms in \eqref{eq_1_nonlin}–\eqref{eq_2_nonlin}, which are sufficient to establish exponential convergence toward equilibrium.

Under these assumptions, the model satisfies a natural conservation property reflecting the reversible nature of the chemical reaction:
\begin{equation*}
    \frac{\dd}{\dd t}\int_{\T\times\R}(f_1-f_2)\,\dd v\,\dd x=0.
\end{equation*}
This motivates the introduction of the unique constant $\rho^\infty>0$ such that
\begin{equation}\label{def_rhoinf}
    \int_{\T\times\R}(f_1^\tin-f_2^\tin)\,\dd v\,\dd x=|\T|\left(\rho^\infty-\frac{1}{\rho^\infty}\right),
\end{equation}
and of the corresponding equilibrium state $\Fbf^\infty=(f_1^\infty,f_2^\infty)^\top$, depending only on velocity, defined by
\begin{equation}\label{def_eq}
    f_1^\infty(x,v)\coloneq\rho_1^\infty\,\chi_1(v),\qquad f_2^\infty(x,v)\coloneq\rho_2^\infty\,\chi_2(v),
\end{equation}
with $\rho_1^\infty = \rho^\infty$ and $\rho_2^\infty = 1/\rho^\infty$.

The exponential convergence to equilibrium for the solution of \eqref{eq_1_nonlin}–\eqref{eq_2_nonlin} was first established in \cite{NeumannSchmeiser2016}. This result relies on a perturbative approach around equilibrium: a hypocoercivity result for the linearized system is first obtained in a weighted $L^2$ framework, following the method of \cite{DolbeaultMouhotSchmeiser2015}, and then extended to the nonlinear system through the propagation of $L^\infty$ bounds. More recently, the authors of \cite{FavrePirnerSchmeiser2023} removed the small-perturbation assumption by exploiting the nonlinear Boltzmann entropy of the system, allowing the treatment of more general initial data that satisfy weaker conditions relative to the equilibrium.

Motivated by these developments, we observed that the Boltzmann entropy $H(\Fbf)$ of the solution to the numerical scheme studied in \cite{BCLR2025}---for which exponential convergence to equilibrium was proven in a perturbative framework---also exhibits exponential decay in time (see Figure~\ref{fig:decayEntropy}). The goal of the present work is to provide a rigorous proof of this decay by adapting the approach of \cite{FavrePirnerSchmeiser2023} to the discrete setting. This extends our previous results to initial data that are not necessarily close to equilibrium, requiring only suitable $L^\infty$ bounds relative to the equilibrium profiles.
\begin{figure}
    \centering
    \includegraphics[width=.6\linewidth]{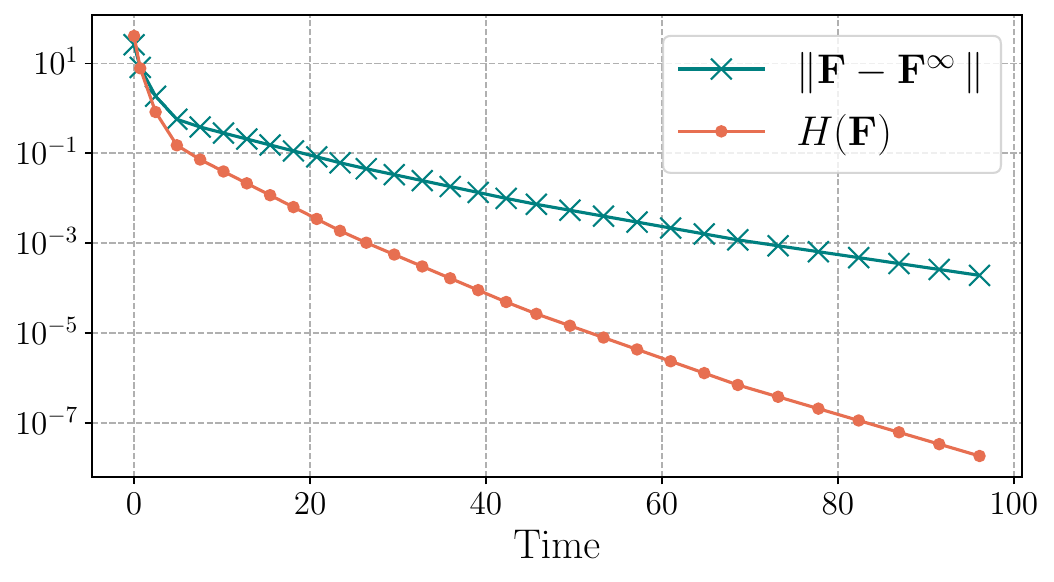}
    \caption{Comparison between the exponential decay of the relative Boltzmann entropy $H$ and of the weighted $L^2$ distance to equilibrium for the numerical scheme studied in \cite{BCLR2025}.}
    \label{fig:decayEntropy}
\end{figure}


The long-time dynamics of \eqref{eq_1_nonlin}--\eqref{eq_2_nonlin} results from the interplay between the transport operator $v\dx$ and the reaction term, which induces a mixing effect in phase space. Under spatial periodic boundary conditions, this mixing drives the system toward the position-independent equilibrium $\Fbf^\infty$. The study of the long-time asymptotics in kinetic equations is a classical topic, dating back to the seminal work \cite{Hoermander1967} on hypoellipticity for linear operators. Early results on the asymptotic behavior of kinetic equations were obtained in \cite{BensoussanLionsPapanicolaou1979}. More recently, for a broad class of collision operators \cite{HerauNier2004, Herau2017, DolbeaultMouhotSchmeiser2015, BouinDolbeaultMischlerMouhotSchmeiser2020}, it has been shown that, in a suitable norm, solutions converge exponentially fast toward equilibrium: there exist constants $\kappa > 0$ and $C > 1$ such that
\begin{equation}
    \left\|\Fbf(t) - \Fbf^\infty\right\|_{\mathcal{X}}
    \le C\left\|\Fbf^{\tin} - \Fbf^\infty\right\|_{\mathcal{X}}
    \e^{-\kappa t},
\end{equation}
for an appropriate Hilbert space $\mathcal{X}$. This property is known as \emph{hypocoercivity}, following \cite{VillaniHypo2009}, whereas the special case $C = 1$ corresponds to the classical notion of coercivity.

Building on the general framework of hypocoercivity, the seminal work \cite{DolbeaultMouhotSchmeiser2015} introduced a robust and versatile method to establish exponential convergence to equilibrium. The core idea is to construct a \emph{modified entropy functional}, equivalent to a weighted $L^2$-norm, for which exponential decay can be proven. This approach has since been successfully applied in a variety of kinetic contexts, including chemotaxis models \cite{CalvezRaoulSchmeiser2015}, the fractional Fokker–Planck equation \cite{BouinDolbeaultLafleche2022}, the Vlasov–Fokker–Planck equation \cite{BouinDolbeaultZiviani2023, BouinZiviani2025}, and the Vlasov–Poisson–Fok\-ker–Planck system \cite{Addala2021}. Extension to control problems in plasma theory, using refined estimates on the trajectories of the solutions, was also considered in \cite{HeHuMo2022}.

More recently, in \cite{FavrePirnerSchmeiser2023}, the authors adapted this strategy to the nonlinear framework. Indeed, instead of relying on the evolution of a weighted $L^2$-norm, they construct a modified entropy functional directly from the nonlinear Boltzmann entropy, which can be shown to decay exponentially. Under an $L^\infty$ assumption on the initial data relative to the equilibrium profiles, the logarithmic terms appearing in the Boltzmann entropy can be carefully linearized, thereby quantifying the distance to equilibrium in the standard weighted $L^2$-norm.

The entropy-based and hypocoercive frameworks developed at the continuous level have also inspired a broad range of studies on their discrete counterparts, aiming to understand how numerical schemes reproduce long-time asymptotics. Over the past decades, the long-time behavior of discretizations of macroscopic models has been extensively studied using energy-based techniques. These methods rely on preserving the entropic structure of the continuous equations at the discrete level \cite{GosseToscani2006, ChainaisFilbet2007, BurgerCarrilloWolfram2010, BessemoulinChatardFilbet2012}. Among them, the $\phi$-entropy method has proven to be a particularly robust tool for analyzing asymptotic behavior \cite{FilbetHerda2017, ChainaisHillairetHerda2020}. The use of $\phi$-entropies was later extended to linear kinetic equations, such as the Fokker–Planck equation \cite{DolbeaultLi2018} and the linear relaxation model \cite{Evans2021}, demonstrating the effectiveness of nonlinear techniques in deriving quantitative convergence rates toward equilibrium.

These ideas naturally extend to the discretization of kinetic equations, allowing numerical schemes to accurately capture long-time asymptotics. Early developments in this direction focused on the Kolmogorov equation, where collisions are modelled by a Laplacian in velocity. A finite-difference scheme was analyzed in \cite{PorrettaZuazua2017}, while finite-element approaches were investigated in \cite{Georgoulis2021, ZhaonanGeorgoulisHerbert2025pp}. These discrete hypocoercivity results led to the design of an efficient stabilization strategy for finite-element methods, as developed in \cite{DongEmmanuilGeorgoulisHerbert2025}. The original $H^1$-hypocoercivity framework “à la Villani” \cite{VillaniHypo2009} was first adapted to the discrete setting through finite-difference discretizations of the kinetic Fokker–Planck equation \cite{DujardinHerauLaffitte2018}. Subsequently, the robust weighted $L^2$ approach of \cite{DolbeaultMouhotSchmeiser2015} was extended to various discrete schemes for linear kinetic models. A discrete hypocoercivity property was first established for finite-volume discretizations of the Fokker–Planck and linear BGK equations in \cite{BessemoulinHerdaRey2020}, and later generalized to the nonlinear model \eqref{eq_1_nonlin}–\eqref{eq_2_nonlin} within a perturbative regime in \cite{BCLR2025}. Long-time asymptotics for a finite-volume scheme of the fractional Fokker–Planck equation were also sudied in \cite{AyiHerdaHivertTristani2023}. Finally, this methodology has been applied to collisional Vlasov equations—both with external potentials and for the linearized Vlasov–Poisson system—via spectral decompositions based on Hermite polynomials \cite{BlausteinFilbet2024_1, BlausteinFilbet2024_2}. We also mention the recent work \cite{Grosse2025}, which employs a fully spectral phase-space decomposition with a linearized BGK operator preserving the first three moments in an unbounded domain with a confining potential. This approach follows the analysis of \cite{CarrapatosoDolbeaultHerauMischlerMouhotSchmeiser2021}, where convergence toward specific macroscopic modes (both stationary and time-periodic) was rigorously established.

\medskip
\textbf{Outline of the paper.}
Section~\ref{sec:ContinuousSetting} recall the main arguments of the continuous analysis, reformulated in a form suitable for the discrete framework. Section~\ref{sec:DiscreteSetting} introduces the numerical scheme, together with its discrete maximum principle and an existence result. The maximum principle follows from a monotonicity property of the chosen upwind numerical fluxes. Finally, Section~\ref{sec:Hypoco} establishes the exponential convergence of the discrete solution toward equilibrium by adapting the continuous proof. In this section, we focus on the specific challenges introduced by the discretization---such as the treatment of additional terms, the lack of a discrete chain rule, and the impact of time discretization---while the algebraic computations from the continuous setting carry over directly and require only a brief discussion.

\section{The continuous setting}\label{sec:ContinuousSetting}
This works builds on and extends our previous contribution \cite{BCLR2025}. In particular, it relies on the recent work \cite{FavrePirnerSchmeiser2023}, adopting a fully nonlinear approach rather than relying on a linearization of the model.  For clarity and readability, we present here a condensed version of the proof from \cite{FavrePirnerSchmeiser2023}, reformulated in a way that facilitates its adaptation to the discrete setting.

First of all, let us recall the global existence result and maximum principle estimates stated in \cite[Theorem 1]{FavrePirnerSchmeiser2023}.
\begin{Proposition}[Global existence and maximum principle]
    Under assumptions \eqref{hyp_chi} on $\chi_k$, we further assume that the initial data $\Fbf^{in}$ satisfies for all $x\in\T$, $v\in\R$, 
    \begin{align}
        &\rho_m\chi_1(v)\leq f_1^{in}(x,v)\leq \rho_M\chi_1(v), \label{hyp.CI.1}\\
        &\rho_M^{-1}\chi_2(v)\leq f_2^{in}(x,v)\leq \rho_m^{-1}\chi_2(v), \label{hyp.CI.2}
    \end{align}
    with constants $0<\rho_m\leq \rho_M<\infty$.
    
    Then there exists a unique global solution $\Fbf$ of \eqref{eq_1_nonlin}--\eqref{eq_2_nonlin} such that for all $x\in\T$, $v\in\R$, $t\geq 0$,
    \begin{align}
        &\rho_m\chi_1(v)\leq f_1(t,x,v)\leq \rho_M\chi_1(v), \label{estimLinf_f1}\\
        &\rho_M^{-1}\chi_2(v)\leq f_2(t,x,v)\leq \rho_m^{-1}\chi_2(v). \label{estimLinf_f2}
    \end{align}
\end{Proposition}

\subsection{Notations and functional setting} 
We now introduce the notations and functional framework that will be used throughout the analysis, providing the tools needed to handle both the continuous and discrete settings.

The decay towards equilibrium will be estimated quantitatively in the following weighted $L^2$ space
\[ \mathcal{X}\coloneq L^2(\T \times\R,\dd x\dd v/f_1^\infty)\times L^2(\T\times\R,\dd x\dd v/f_2^\infty),  \]
endowed with the scalar product
\[ \lla \Fbf,\Gbf\rra = \int_{\T\times\R} \left(\frac{f_1g_1}{f_1^\infty}+\frac{f_2g_2}{f_2^\infty}\right)\dd x \dd v,\quad\forall\, \Fbf=(f_1,f_2)^\top,\,\Gbf= (g_1,g_2)^\top\in\mathcal{X}.\]
The corresponding norm is denoted by $\|\cdot\|$.

We also introduce the space
\[ L^2(\T, \dd x/\rho_1^\infty)\times L^2(\T, \dd x/\rho_2^\infty),  \]
equipped with the scalar product $\lla \cdot,\cdot\rra_x$ and the associated norm $\|\cdot\|_x$.

We define the orthogonal projection onto the velocity equilibrium by
\[\Pi \Fbf\coloneq \begin{pmatrix}\rho_1\,\chi_1 \\ \rho_2\,\chi_2\end{pmatrix},\]
where $\rho_k$ is defined by~\eqref{def_densities}.

To study the convergence of $\Fbf(t)$ toward $\Fbf^\infty$ as $t$ goes to $\infty$, we introduce the deviation $\Ftbf(t)\coloneq \Fbf(t)-\Fbf^\infty$ and rewrite the system  \eqref{eq_1_nonlin}--\eqref{eq_2_nonlin} as
\begin{align}
    &\dt \ft_1 + v\dx \ft_1 = \chi_1 - \rho_2 f_1, \label{eq_1tilde_nonlin}\\
    &\dt \ft_2 + v\dx \ft_2 = \chi_2 - \rho_1 f_2. \label{eq_2tilde_nonlin}
\end{align}

\subsection{Moment dynamics and estimates}
To analyze the long-time behavior of the system, it is useful to consider its macroscopic quantities \cite{DolbeaultMouhotSchmeiser2015}. These moments capture essential features of the dynamics and play a key role in establishing coercivity properties. We now define these moments and provide corresponding estimates with respect to the norm $\|\cdot\|$. The moments of $\Fbf=\left(f_1, f_2\right)^\top$ are defined as
\begin{align}
 &   \rhobf_\Fbf\coloneq (\rho_1,\rho_2)^\top,\quad \text{with } \rho_k=\int_\R f_k\dd v, \label{def.rhoF}\\
 &   \Jbf_\Fbf\coloneq (J_1,J_2)^\top,\quad \text{with } J_k=\int_\R v\,f_k\dd v, \label{def.JF}\\
 &   \Sbf_\Fbf\coloneq (S_1,S_2)^\top,\quad \text{with } S_k=\int_\R (v^2-D_k)f_k\dd v, \label{def.SF}
\end{align}
$D_k$ being the second-order moment of $\chi_k$, defined in \eqref{hyp_chi}. Let us remark that with these notations,
\[ 
\rhotbf\coloneq\rhobf_\Ftbf=\rhobf_\Fbf-\rhobf^\infty,\quad \Jtbf\coloneq\Jbf_\Ftbf=\Jbf_\Fbf,\quad \Stbf\coloneq\Sbf_\Ftbf=\Sbf_\Fbf.
\]
By multiplying \eqref{eq_1tilde_nonlin} and \eqref{eq_2tilde_nonlin} by $(1,v)^\top$ and integrating with respect to $v$, we get the following moment equations for $\Ftbf$:
\begin{equation}\label{eq.rhoFt}
    \dt \rhotbf+\dx \Jtbf=(1-\rho_1\rho_2)\left(\begin{matrix}1\\ 1\end{matrix}\right),
\end{equation}
and 
\begin{align}
    \dt \Jt_1+\dx \St_1+D_1\dx \rhot_1=-\rho_2J_1,\label{eq.Jt1}\\
    \dt \Jt_2+\dx \St_2+D_2\dx \rhot_2=-\rho_1J_2. \label{eq.Jt2}
\end{align}
The following lemma provides estimates, which are essential for controlling the macroscopic contributions in the subsequent hypocoercivity analysis.

\begin{Lemma}[Moments estimates]
\label{lem_mom_estim}
Under assumptions \eqref{hyp_chi}, the moments satisfy the following estimates
\begin{align}
    & \|\rhotbf\|_x=\|\Pi \Ftbf\|, \label{estim.rhoFt}\\
    & \|\Jtbf\|_x=\|\Jbf_\Fbf\|_x\leq C_1 \,\|(I-\Pi)\Ftbf\|, \label{estim.JFt}\\
    & \|\Stbf\|_x=\|\Sbf_\Fbf\|_x\leq C_2 \,\|(I-\Pi)\Ftbf\|, \label{estim.SFt}
\end{align}
where the constants are given by
\begin{align}
    & C_1 = \sqrt{\max(D_1,D_2)}, \label{def.C1}\\
    & C_2 = \sqrt{\max(Q_1-D_1^2,Q_2-D_2^2)}. \label{def.C2}
\end{align}
\end{Lemma}

\begin{proof}
We start by proving \eqref{estim.rhoFt}. Using the definition of $\Pi$ and the first property in  \eqref{hyp_chi}, one has
\begin{equation*}
    \|\Pi \Ftbf\|^2= \int_{\T\times \R}\left[ \frac{(\rhot_1\chi_1)^2}{\chi_1\rho_1^\infty}+\frac{(\rhot_2\chi_2)^2}{\chi_2\rho_2^\infty} \right]\dd x \dd v=\int_{\T\times\R}\left[\frac{\rhot_1^2\chi_1}{\rho_1^\infty}+\frac{\rhot_2^2\chi_2}{\rho_2^\infty}\right]\dd x\dd v=\|\rhotbf\|^2_x.
\end{equation*}
Now for \eqref{estim.JFt}, we have by definition
\[\|\Jbf_\Fbf\|_x^2=\int_\T\sum_{k=1,2}\frac{1}{\rho_k^\infty}\left( \int_\R v\,f_k\,\dd v\right)^2\dd x.\]
Using the second property in \eqref{hyp_chi}, we get
\begin{align*}
    \left( \int_\R v\,f_k\,\dd v\right)^2&=\left( \int_\R v\,(f_k-\rho_k\chi_k)\frac{\sqrt{\chi_k}}{\sqrt{\chi_k}}\,\dd v\right)^2\\
    &\leq \left( \int_\R \frac{(f_k-\rho_k\chi_k)^2}{\chi_k}\dd v \right)\left(\int_\R v^2\chi_k\dd v \right),
\end{align*}
which yields the result using the definition of $D_k$ given in \eqref{hyp_chi}.

Finally, for \eqref{estim.SFt}, we have by definition
\[ \|\Sbf_\Fbf\|_x^2=\sum_{k=1,2}\int_\T\frac{1}{\rho_k^\infty}\left( \int_\R (v^2-D_k)f_k\,\dd v \right)^2\dd x.\]
Now, using the definition of $D_k$, we have $\int_\R (v^2-D_k)\rho_k\chi_k\dd v = 0$, which gives
\begin{align*}
    \left( \int_\R (v^2-D_k)f_k\,\dd v \right)^2&= \left( \int_\R (v^2-D_k)(f_k-\rho_k\chi_k)\frac{\sqrt{\chi_k}}{\sqrt{\chi_k}}\,\dd v \right)^2\\
    &\leq \left( \int_\R (v^2-D_k)^2\chi_k \dd v\right)\left(\int_\R \frac{(f_k-\rho_k\chi_k)^2}{\chi_k}\dd v \right),
\end{align*}
yielding the result using properties \eqref{hyp_chi} of $\chi_k$.
\end{proof}

\subsection{Boltzmann entropy dissipation}
We define the relative Boltzmann entropy of $\Fbf=(f_1,f_2)^\top$ with respect to the equilibrium $\Fbf^\infty=(f_1^\infty,f_2^\infty)^\top$ as
\begin{equation}\label{def.Boltzmann}
    H(\Fbf)\coloneq  \sum_{k=1,2}\int_{\T\times\R}\left( f_k\left[ \log\left( \frac{f_k}{f_k^\infty}\right)-1 \right] +f_k^\infty\right)\dd v\dd x.
\end{equation}
The first step of the analysis is to establish the following microscopic coercivity property, which provides dissipation a priori solely in the velocity variable.

\begin{Lemma}[Microscopic coercivity]\label{lem_microcoercivity}
Let \eqref{hyp_chi} hold and let $\Fbf=(f_1,f_2)^\top$ be the solution to the system \eqref{eq_1_nonlin}--\eqref{eq_2_nonlin} with initial data $\Fbf^{in}=(f_1^{in},f_2^{in})^\top \in \mathcal{X}$. Then for every $t\geq 0$,
\begin{equation}\label{estim.H_D}
    \frac{\dd}{\dd t}H(\Fbf)=-\mathcal{D}(\Fbf)\leq 0,
\end{equation}
where the dissipation is defined by
\begin{equation}\label{def.D}
    \mathcal{D}(\Fbf)\coloneq \int_{\T\times\R^2}(f_1f_2'-\chi_1\chi_2')\log\left( \frac{f_1f_2'}{\chi_1\chi_2'} \right)\dd v \dd v' \dd x.
\end{equation}
\end{Lemma}

\begin{proof}
Since $\Fbf=(f_1,f_2)^\top$ satisfies \eqref{eq_1_nonlin}--\eqref{eq_2_nonlin}, we have
\begin{align*}
    \frac{\dd}{\dd t}H(\Fbf)&=\sum_{k=1,2}\int_{\T\times\R}\dt f_k\log\left( f_k/f_k^\infty\right)\dd v\dd x\\
    &=\int_{\T\times\R}(-v\dx f_1+\chi_1-\rho_2f_1)\log\left( f_1/f_1^\infty\right)\dd v\dd x\\
    &\,+ \int_{\T\times\R}(-v\dx f_2+\chi_2-\rho_1f_2)\log\left( f_2/f_2^\infty\right)\dd v\dd x.
\end{align*}
Performing an integration by parts with respect to $x$ and using the periodic boundary conditions, we remark that the transport terms vanish:
\begin{equation*}
    \begin{aligned}
        -\int_{\T\times\R}v\dx f_k\log\left( f_k/f_k^\infty\right)\dd v\dd x&=\int_{\T\times\R}vf_k\dx\left(\log\left( f_k/f_k^\infty\right)\right)\dd v\dd x\\ &=\int_{\T\times\R}v\dx f_k\dd v\dd x=0.
    \end{aligned}
\end{equation*}
It remains to treat the relaxation terms. To do so, we double the velocity variables using that $\int_\R \chi_k'\dd v'=1$ and $\rho_k=\int_\R f_k'\dd v'$:
\begin{gather*}
    \int_{\T\times\R}\left[(\chi_1-\rho_2f_1)\log\left( \frac{f_1}{\rho_1^\infty\chi_1}\right)+(\chi_2-\rho_1f_2)\log\left( \frac{f_2}{\rho_2^\infty\chi_2}\right)\right]\dd v\dd x=\\
    \int_{\T\times\R^2}\left[(\chi_1\chi_2'-f_2'f_1)\log\left( \frac{f_1}{\rho_1^\infty\chi_1}\right)+(\chi_2\chi_1'-f_1'f_2)\log\left( \frac{f_2}{\rho_2^\infty\chi_2}\right)\right]\dd v'\dd v\dd x.
\end{gather*}
By reversing the integration variables $v$ and $v'$ in the second term, we get
\[\frac{\dd}{\dd t}H(\Fbf)=\int_{\T\times\R^2}(\chi_1\chi_2'-f_1f_2')\left[\log\left(\frac{f_1f_2'}{\chi_1\chi_2'}\right)-\log(\rho_1^\infty\rho_2^\infty)\right]\dd v'\dd v\dd x,\]
which concludes the proof using that $\rho_1^\infty\rho_2^\infty=1$.
\end{proof}

We also require the following technical result to establish exponential convergence toward equilibrium. It connects the dissipation to the nonlinear reaction terms in the moment equations \eqref{eq.Jt1}--\eqref{eq.Jt2}, as well as to the dissipation in the direction of $I-\Pi$, which corresponds to dissipation in the velocity variable.

\begin{Lemma}\label{lem_estimD}
Under assumptions \eqref{hyp_chi} and \eqref{hyp.CI.1}--\eqref{hyp.CI.2}, the dissipation term defined by \eqref{def.D} satisfies
\begin{equation}\label{estim.D}
    \mathcal{D}(\Fbf)\geq C_3\|1-\rho_1\rho_2\|_{L^2(\T)}^2+C_4\|(I-\Pi)\Ftbf\|^2,
\end{equation}
where $C_3=\rho_m/\rho_M$ and $C_4=(\rho_m/\rho_M)(\rho_1^\infty+\rho_2^\infty)^{-1}\min(\rho_m,\rho_M^{-1})^2$.
\end{Lemma}

\begin{proof}
Using Taylor expansion, we have
\[\log\left(\frac{f_1f_2'}{\chi_1\chi_2'}\right)=\frac{1}{\widehat{f_1f_2'}}(f_1f_2'-\chi_1\chi_2'),\]
where $\widehat{f_1f_2'} \in \left(\min(f_1f_2', \chi_1\chi_2'),\max(f_1f_2', \chi_1\chi_2')\right)$. In particular $\widehat{f_1f_2'}\leq (\rho_M/\rho_m)\chi_1\chi_2'$ by assumptions \eqref{hyp.CI.1}--\eqref{hyp.CI.2}, and it yields
\begin{equation}\label{estim.Da}
    \mathcal{D}(\Fbf)\geq \frac{\rho_m}{\rho_M}\int_{\T\times\R^2}\frac{(f_1f_2'-\chi_1\chi_2')^2}{\chi_1\chi_2'}\dd v'\dd v\dd x\eqcolon\frac{\rho_m}{\rho_M}\mathcal{E}(\Fbf).
\end{equation}
Moreover, using that 
\[\rho_1^\infty+\rho_2^\infty=\frac{1}{\rho_1^\infty}+\frac{1}{\rho_2^\infty}\]
and $\int_\R \chi_k'\dd v'=1$, we remark that 
\begin{gather*}
    (\rho_1^\infty+\rho_2^\infty)\mathcal{E}(\Fbf)=\\
    \int_{\T\times\R}\frac{1}{\rho_1^\infty}\int_\R\frac{(f_1f_2'-\chi_1\chi_2')^2}{\chi_1\chi_2'}\dd v'\dd v\dd x+ \int_{\T\times\R}\frac{1}{\rho_2^\infty}\int_\R\frac{(f_1'f_2-\chi_1'\chi_2)^2}{\chi_1'\chi_2}\dd v'\dd v\dd x\\
    = \sum_{\substack{k,l=1,2\\ k\neq l}}\int_{\T\times\R}\frac{1}{\rho_k^\infty\chi_k}\left[\left(\int_\R\frac{f_kf_l'-\chi_k\chi_l')^2}{\chi_l'}\dd v'\right)^{\frac{1}{2}}\left(\int_\R \chi_l'\dd v'\right)^{\frac{1}{2}}\right]^2.
\end{gather*}
Then, thanks to the Cauchy-Schwarz inequality, we get
\begin{gather*}
(\rho_1^\infty+\rho_2^\infty)\mathcal{E}(\Fbf) \geq \sum_{\substack{k,l=1,2\\ k\neq l}}\int_{\T\times\R}\frac{1}{\rho_k^\infty\chi_k}\left[\int_\R(f_kf_l'-\chi_k\chi_l')\dd v'\right]^2\dd v\dd x\\
\geq \sum_{\substack{k,l=1,2\\ k\neq l}}\int_{\T\times\R}\frac{1}{\rho_k^\infty\chi_k}(f_k\rho_l-\chi_k)^2\dd v\dd x\\
\geq \sum_{\substack{k,l=1,2\\ k\neq l}}\int_{\T\times\R}\frac{1}{\rho_k^\infty\chi_k}\left[\rho_l(f_k-\rho_k\chi_k)-\chi_k(1-\rho_k\rho_l)\right]^2\dd v\dd x.
\end{gather*}
Noticing that the cross term in the expansion of the square satisfies
\[ \int_{\T\times\R}\frac{1}{\rho_k^\infty\chi_k}\rho_l(f_k-\rho_k\chi_k)\chi_k(1-\rho_k\rho_l)\dd v\dd x =0,\]
we obtain
\begin{gather*}
(\rho_1^\infty+\rho_2^\infty)\mathcal{E}(\Fbf) \geq  \sum_{\substack{k,l=1,2\\ k\neq l}}\int_{\T\times\R}\frac{1}{\rho_k^\infty\chi_k}\left(\rho_l^2(f_k-\rho_k\chi_k)^2+\chi_k^2(1-\rho_k\rho_l)^2\right)\dd v\dd x\\
\geq  \sum_{\substack{k,l=1,2\\ k\neq l}}\int_{\T\times\R}\frac{1}{f_k^\infty}\rho_l^2(f_k-\rho_k\chi_k)^2\dd v\dd x+\left(\frac{1}{\rho_1^\infty}+\frac{1}{\rho_2^\infty}\right)\int_\T (1-\rho_1\rho_2)^2\dd x \\
\geq \min(\rho_m,\rho_M^{-1})^2 \sum_{\substack{k,l=1,2\\ k\neq l}}\int_{\T\times\R}\frac{1}{f_k^\infty}(f_k-\rho_k\chi_k)^2\dd v\dd x+(\rho_1^\infty+\rho_2^\infty)\|1-\rho_1\rho_2\|_{L^2(\T)}^2,
\end{gather*}
where we have used $\rho_1^2\geq \rho_m^2$ and $\rho_2^2\geq \rho_M^{-2}$ in the last inequality. Finally, by definition of $I-\Pi$ and since $\Fbf^\infty$ belongs to the nullspace of $I-\Pi$, we get
\begin{align*}
(\rho_1^\infty+\rho_2^\infty)\mathcal{E}(\Fbf) &\geq \min(\rho_m,\rho_M^{-1})^2\|(I-\Pi)F\|^2+(\rho_1^\infty+\rho_2^\infty)\|1-\rho_1\rho_2\|_{L^2(\T)}^2\\
&\geq \min(\rho_m,\rho_M^{-1})^2\|(I-\Pi)\Ft\|^2+(\rho_1^\infty+\rho_2^\infty)\|1-\rho_1\rho_2\|_{L^2(\T)}^2,
\end{align*}
which concludes the proof by combining this estimate with \eqref{estim.Da}.
\end{proof}

\subsection{Nonlinear hypocoercivity}
With the microscopic and macroscopic estimates established in the previous subsections, we are now in a position to state the main hypocoercivity result given in \cite{FavrePirnerSchmeiser2023}. This theorem quantifies the exponential convergence of the solution toward equilibrium, combining the velocity dissipation and the macroscopic coercivity into a unified nonlinear estimate.
\begin{Theorem}\label{thm.hypoco}
    Under assumptions \eqref{hyp_chi}, there exist two constants $C\geq 1$ and $\kappa>0$ such that for all initial data $\Fbf^{in}=(f_1^{in},f_2^{in})^\top \in \mathcal{X}$ fulfilling \eqref{hyp.CI.1}--\eqref{hyp.CI.2}, the solution $\Fbf=(f_1,f_2)^\top$ of \eqref{eq_1_nonlin}--\eqref{eq_2_nonlin} satisfies
    \begin{equation}\label{estim.hypoco}
        \|\Fbf(t)-\Fbf^\infty\|\leq C\|\Fbf^{in}-\Fbf^\infty\|e^{-\kappa t}\quad \forall t\geq 0.
    \end{equation}
\end{Theorem}

In order to prove Theorem~\ref{thm.hypoco}, we introduce a modified entropy functional, following the approach of \cite{FavrePirnerSchmeiser2023,DolbeaultMouhotSchmeiser2015}:
\begin{equation}\label{def.Gamma}
    \Gamma(t)\coloneq H(\Fbf)+\delta\lla \Jtbf,\dx\Phibf\rra_x,
\end{equation}
where $\Phibf=\left(\phi_1, \phi_2\right)^\top$, with each $\phi_k$ solving the Poisson equation
\begin{equation}\label{eq.Phi}
    -\partial_{xx}\phi_k=\rhot_k,\quad \int_\T \phi_k\dd x=0,\quad k=1,2,
\end{equation}
and $\delta>0$ is a small parameter to be fixed later. The additional term in \eqref{def.Gamma} is designed to explicitly recover dissipation in the spatial variable, reflecting the mixing effect induced by the transport operator.

We begin with key estimates on $\Phibf$ that will be used to control the additional term in the modified entropy $\Gamma(t)$.
\begin{Lemma}
    The function $\Phibf$ satisfies for all $t\geq 0$
    \begin{align}
        &\|\dx\Phibf\|_x\leq C_5\|\Pi \Ftbf\|, \label{estim.dxphi}\\
        &\|\dt\dx\Phibf\|_x\leq C_1\|(I-\Pi)\Ftbf\|+C_6\|1-\rho_1\rho_2\|_{L^2(\T)}, \label{estim.dtdxphi}
    \end{align}
    where $C_5$ is the Poincaré constant of $\T$ and $C_6=C_5\sqrt{\rho_1^\infty+\rho_2^\infty}$.
\end{Lemma}
\begin{proof}
    The first estimate is obtained by using the Poisson equation \eqref{eq.Phi}, applying the Poincaré inequality and equality \eqref{estim.rhoFt}:
    \[ \|\dx\Phibf\|_x^2=\lla\Phibf,-\partial_{xx}\Phibf\rra_x=\lla\Phibf,\rhotbf\rra_x\leq \|\Phibf\|_x\|\rhotbf\|_x\leq C_5\|\dx\Phibf\|_x\|\Pi\Ftbf\|.\]
    For the second estimate, we differentiate the Poisson equation with respect to time and use the continuity equation \eqref{eq.rhoFt} to get
    \begin{align*}
        \|\dt\dx\Phibf\|_x^2&=\lla \dt\Phibf,-\dt\partial_{xx}\Phibf\rra_x=\lla\dt\Phibf,\dt\rhotbf\rra_x\\
        &=\lla\dt\Phibf,-\dx \Jtbf+(1-\rho_1\rho_2)\left(\begin{matrix}1\\1\end{matrix}\right)\rra_x.
    \end{align*}
    An integration by parts on the first term and the Cauchy-Schwarz inequality give
    \begin{align}
        \|\dt\dx\Phibf\|_x^2&=\lla\dt\dx\Phibf,\Jtbf\rra_x+\lla\dt\Phibf,(1-\rho_1\rho_2)\left(\begin{matrix}1\\ 1\end{matrix}\right)\rra_x \nonumber\\
        &\leq \|\dt\dx\Phibf\|_x\|\Jtbf\|_x+\|\dt\Phibf\|_x\left\|\left(\begin{matrix}1-\rho_1\rho_2\\1-\rho_1\rho_2\end{matrix}\right)\right\|_x. \label{estim.dtdxPhiint}
    \end{align}
    Using that $\rho_1^\infty\rho_2^\infty=1$, we have
    \begin{equation*}
        \left\|\left(\begin{matrix}1-\rho_1\rho_2\\1-\rho_1\rho_2\end{matrix}\right)\right\|_x^2=\left(\frac{1}{\rho_1^\infty}+\frac{1}{\rho_2^\infty}\right)\int_\T(1-\rho_1\rho_2)^2\dd x=(\rho_1^\infty+\rho_2^\infty)\|1-\rho_1\rho_2\|_{L^2(\T)}^2.
    \end{equation*}
    Using this together with the Poincaré inequality on the second term of \eqref{estim.dtdxPhiint}, and the estimate \eqref{estim.JFt} on the first term of \eqref{estim.dtdxPhiint} finally gives the result.
\end{proof}

Using the $L^\infty$ bounds \eqref{estimLinf_f1}--\eqref{estimLinf_f2} on $\Fbf$ together with the moment estimates, we show that, for $\delta>0$ sufficiently small, the modified entropy $\Gamma(t)$ defines a norm equivalent to $\|\cdot\|$.

\begin{Lemma}[Equivalent norm]\label{lem.normeq}
    There is $\delta_1>0$ such that for all $\delta\in (0,\delta_1)$, there are positive constants $0<c_\delta<C_\delta$ such that if $\Fbf\in\mathcal{X}$ satisfies \eqref{estimLinf_f1}--\eqref{estimLinf_f2}, one has
    \[c_\delta\|\Ftbf(t)\|^2\leq \Gamma(t)\leq C_\delta\|\Ftbf(t)\|^2.\] 
\end{Lemma}
\begin{proof}
    We first treat the Boltzmann part of $\Gamma$. Applying a Taylor expansion to the function 
    \[ \xi\mapsto \xi\left(\log\left(\xi/f_k^\infty\right)-1\right)+f_k^\infty, \]
    there exists $\xi_k \in \left ( \min(f_k,f_k^\infty),\max(f_k,f_k^\infty)\right)$ such that
    \[f_k\left[\log\left(f_k/f_k^\infty\right)-1\right]+f_k^\infty=\frac{1}{2\xi_k}(f_k-f_k^\infty)^2.\]
    Since $f_k^\infty=\chi_k\rho_k^\infty$ and given the bounds \eqref{estimLinf_f1}--\eqref{estimLinf_f2} on $f_k$, it holds 
    \begin{gather*}
        \frac{\rho_m}{\rho_1^\infty}f_1^\infty\leq \xi_1\leq\frac{\rho_M}{\rho_1^\infty}f_1^\infty,\\
        \frac{1}{\rho_M\rho_2^\infty}f_2^\infty\leq \xi_2\leq \frac{1}{\rho_m\rho_2^\infty}f_2^\infty.
    \end{gather*}
    We then have
    \begin{align*}
        H(\Fbf)&=\int_{\T\times\R}\left[\frac{(f_1-f_1^\infty)^2}{2\xi_1}+\frac{(f_2-f_2^\infty)^2}{2\xi_2}\right]\dd v \dd x\\
        &\leq \int_{\T\times\R}\left[\frac{1}{2}(f_1-f_1^\infty)^2\frac{\rho_1^\infty}{\rho_m f_1^\infty}+\frac{1}{2}(f_2-f_2^\infty)^2\frac{\rho_M\rho_2^\infty}{f_2^\infty}\right] \dd v \dd x\\
        &\leq \frac{1}{2}\max \left(\frac{\rho_1^\infty}{\rho_m},\rho_M\rho_2^\infty\right)\|\Ftbf\|^2\eqcolon C_H\|\Ftbf\|^2.
    \end{align*}
    Proceeding in the same way for the lower bound, we obtain
    \begin{equation*}
        H(\Fbf)\geq \frac{1}{2}\min\left(\frac{\rho_1^\infty}{\rho_M},\rho_m\rho_2^\infty\right)\|\Ftbf\|^2\eqcolon c_H\|\Ftbf\|^2.
    \end{equation*}
    We now treat the second term of $\Gamma$:
    \[
    \left|\lla \Jtbf,\dx\Phibf\rra_x\right|\leq \|\Jtbf\|_x\|\dx\Phibf\|_x\leq C_1\|(I-\Pi)\Ftbf\|C_5\|\Pi\Ftbf\|\leq C_1\,C_5\|\Ftbf\|^2.
    \]
    This finally gives the expected result with $c_\delta=c_H-\delta C_1C_5$ and $C_\delta=C_H+\delta C_1C_5$. One concludes the proof by choosing $\delta>0$ such that $c_\delta>0$, that is $\delta <\delta_1\coloneq  C_H/(C_1C_5)$.
\end{proof}

Combining lemmas~\ref{lem_estimD} and \ref{lem.normeq}, one can then establish the proof of Theorem~\ref{thm.hypoco}.

\begin{proof}[Proof of Theorem \ref{thm.hypoco}]
    Using the microscopic coercivity \eqref{estim.H_D}, we have
    \[ \frac{\dd}{\dd t}\Gamma(t)=-\mathcal{D}(\Fbf)+\delta\frac{\dd}{\dd t}\lla \Jtbf,\dx\Phibf\rra_x. \]
    We start by studying the second term. Using the continuity equations \eqref{eq.Jt1}--\eqref{eq.Jt2}, one has
    \[\frac{\dd}{\dd t}\lla \Jtbf,\dx\Phibf\rra_x=\lla\dt \Jtbf,\dx\Phibf\rra_x+\lla \Jtbf,\dt\dx\Phibf\rra_x= T_1+T_2+T_3+T_4,\]
    where
    \begin{align*}
        T_1 &\coloneq  -\lla\dx \Stbf,\dx\Phibf\rra_x,\\
        T_2 &\coloneq  -\lla\left(\begin{matrix}D_1\dx\rhot_1 \\ D_2\dx\rhot_2\end{matrix}\right),\dx\Phibf\rra_x,\\
        T_3 &\coloneq  -\lla\left(\begin{matrix} \rho_2 J_1\\ \rho_1 J_2\end{matrix}\right),\dx\Phibf\rra_x,\\
        T_4 &\coloneq \lla \Jtbf,\dt\dx\Phibf\rra_x.
    \end{align*}
    The term $T_1$ can be easily controlled with moment estimates \eqref{estim.rhoFt} and \eqref{estim.SFt}:   
    \begin{equation*}
        |T_1|=\left|\lla \Stbf,\partial_{xx}\Phibf\rra_x\right|=\left|-\lla \Stbf,\rhotbf\rra_x\right|\leq \|\Stbf\|_x\|\rhotbf\|_x\leq C_2\|(I-\Pi)\Ftbf\|\|\Pi \Ftbf\|.
    \end{equation*}
    The term $T_2$ provides the macroscopic coercivity. Indeed, using an integration by parts and the Poisson equation \eqref{eq.Phi}, we get
    \begin{align*}
        T_2&=-\sum_{k=1,2}\frac{D_k}{\rho_k^\infty}\lla\dx\rhot_k,\dx\phi_k\rra_{L^2(\T)}=\sum_{k=1,2}\frac{D_k}{\rho_k^\infty}\lla \rhot_k,\dxx\phi_k\rra_{L^2(\T)}\\
        &\leq -\min(D_1,D_2)\|\rhotbf\|_x^2 = -\min(D_1,D_2)\|\Pi\Ftbf\|^2.
    \end{align*}
    In the terms $T_3$ and $T_4$, nonlinear contributions appear. These will be controlled on the one hand by the Boltzmann entropy dissipation (microscopic coercivity), and on the other hand by the $L^\infty$ bounds on $\Fbf$. Using the Cauchy-Schwarz inequality, we have
    \[|T_3|\leq \left\|\left(\begin{matrix}\rho_2 J_1\\ \rho_1 J_2\end{matrix}\right)\right\|_x\|\dx\Phibf\|_x.\]
    From the $L^\infty$ bounds on $\Fbf$, we deduce that
    \[ \rho_m\leq \rho_1\leq \rho_M,\quad \rho_M^{-1}\leq \rho_2\leq \rho_m^{-1},\]
    which yields
    \begin{align*}
        \left\|\left(\begin{matrix}\rho_2 J_1\\ \rho_1 J_2\end{matrix}\right)\right\|_x^2&=\frac{1}{\rho_1^\infty}\int_\T(\rho_2 J_1)^2\dd x+\frac{1}{\rho_2^\infty}\int_\T (\rho_1 J_2)^2\dd x\\
        &\leq \frac{\rho_m^{-2}}{\rho_1^\infty}\int_\T J_1^2\dd x+\frac{\rho_M^2}{\rho_2^\infty}\int_\T J_2^2\dd x\\
        &\leq \max(\rho_m^{-1},\rho_M)^2\|\Jbf_\Fbf\|^2_x\\
        & \leq \max(\rho_m^{-1},\rho_M)^2C_1^2\|(I-\Pi)\Ftbf\|^2,
    \end{align*}
    and finally, using also \eqref{estim.dxphi},
    \begin{equation*}
        |T_3|\leq \max(\rho_m^{-1},\rho_M)C_1C_5\|(I-\Pi)\Ftbf\|\|\Pi\Ftbf\|.
    \end{equation*}
    Let us now turn to the estimate of $T_4$. Thanks to the Cauchy-Schwarz inequality and using estimates \eqref{estim.JFt} and \eqref{estim.dtdxphi}, we write
    \begin{align*}
        |T_4|&\leq \|\Jtbf\|_x\|\dt\dx\Phibf\|_x\\
        &\leq C_1\|(I-\Pi)\Ftbf\|\left[C_1\|(I-\Pi)\Ftbf\|+C_6\|1-\rho_1\rho_2\|_{L^2(\T)}\right]\\
        &\leq C_1^2\|(I-\Pi)\Ftbf\|^2+C_1C_6\|(I-\Pi)\Ftbf\|\|1-\rho_1\rho_2\|_{L^2(\T)}.
    \end{align*}
    Using the Young inequality on the second term of the right hand side, we get
    \[|T_4|\leq C_1^2\left(1+\frac{C_6^2}{2}\right)\|(I-\Pi)\Ftbf\|^2+\frac{1}{2}\|1-\rho_1\rho_2\|_{L^2(\T)}^2.\]

    If we now recombine all the terms, and use estimate \eqref{estim.D} of the dissipation $\mathcal{D}(\Fbf)$, we obtain
    \begin{gather*}
        \frac{\dd}{\dd t}\Gamma(t)\leq -\left(C_3-\frac{\delta}{2}\right)\|1-\rho_1\rho_2\|_{L^2(\T)}^2-\left(C_4-\delta C_1^2\left(1+\frac{C_6^2}{2}\right)\right)\|(I-\Pi)\Ftbf\|^2\\
        -\delta\min(D_1,D_2)\|\Pi\Ftbf\|^2+\delta\left(C_2+\max(\rho_m^{-1},\rho_M)C_1C_5\right)\|(I-\Pi)\Ftbf\|\|\Pi\Ftbf\|.
    \end{gather*}
  If $\delta>0$ is chosen sufficiently small, we obtain two constants $K_\delta>0$ and $\overline{K}_\delta>0$ such that
    \[\frac{\dd}{\dd t}\Gamma(t)+K_\delta\|\Ftbf(t)\|^2+\overline{K}_\delta \|1-\rho_1\rho_2\|^2_{L^2(\T)}\leq 0,\]
    and we can conclude by using Lemma \ref{lem.normeq}.
\end{proof}

\section{The discrete setting}\label{sec:DiscreteSetting}
In this section, we recall the numerical scheme introduced in \cite{BCLR2025}. As discussed therein, the use of monotone numerical fluxes is crucial to ensure a discrete maximum principle. In \cite{BCLR2025}, Lax–Friedrichs fluxes were employed, as they are well suited for extending the hypocoercivity framework to the discrete setting. However, these fluxes do not appear to dissipate the Boltzmann entropy--at least no rigorous proof of such a property is available--making them unsuitable for the present analysis. We therefore consider upwind fluxes, which are both monotone and, as will be shown below, compatible with the Boltzmann entropy structure.

\subsection{Notations}
Before introducing the numerical scheme, we fix the notations and describe the underlying mesh.
\subsubsection{Mesh}
Since implementing a numerical scheme on an unbounded domain is not feasible in practice, we first restrict the velocity domain to a bounded and symmetric interval $[-v^*,v^*]$. We consider a discretization of $[-v^*,v^*]$ composed of $2L$ control volumes, arranged symmetrically around the origin. Let $v_\jph$ denote the $2L+1$ interface points, indexed by $j\in\J\coloneq\{-L,\cdots,L\}$, so that
\[ v_{-L+\frac{1}{2}}=-v^*,\quad v_{\frac{1}{2}}=0,\quad v_\jph=-v_\jmh \qquad \forall j=0,\cdots,L.\]
For simplicity, we assume that the velocity grid is uniform, namely that each cell $\mathcal{V}_j=(v_\jmh,v_\jph)$ has the same length $\Dv$. Denoting by $v_j$ the midpoint of the cell $\mathcal{V}_j$, we also have $v_j=-v_{-j+1}$ for all $j=1,\cdots,L$.

In the position variable, we consider a uniform discretization of the torus $\T$ into $N$ control volumes
\[\mathcal{X}_i\coloneq(x_\imh,x_\iph), \quad i\in\II\coloneq\mathbb{Z}/N\mathbb{Z},\]
of constant length $\Dx$. Once again, uniformity of the mesh is assumed for clarity, although the analysis extends to non-uniform meshes. As in \cite{BessemoulinHerdaRey2020}, we further impose that $N$ is odd, which in particular ensures the validity of a discrete Poincaré inequality on the torus for the centered discrete derivative operator.

Finally, the control volumes in phase space are defined by
\[ K_{ij}\coloneq\mathcal{X}_i\times\mathcal{V}_j,\qquad \forall (i,j)\in\II\times \J,\]
and the mesh size is given by $\Delta=(\Dx,\Dv)$.

In addition, we denote by $\Delta t>0$ the time step used for the discrete time integration.

\subsubsection{Discrete velocity profiles} 
We now define discrete versions of the velocity profiles $\chi_k$. Our goal is to ensure that these discrete functions satisfy discrete counterparts of assumptions~\eqref{hyp_chi}, which play a crucial role and are repeatedly used throughout the continuous analysis. 

For $k=1,\,2$, we consider discrete values $(\chi_{k,j})_{j\in\J}\in\mathbb{R}^\J$ that satisfy the following discrete counterparts of assumptions~\eqref{hyp_chi}
\begin{equation}\label{hyp_chi_dis}
    \begin{aligned}
        & \chi_{k,j}>0,\quad \chi_{k,j}=\chi_{k,-j+1} \qquad \forall j=1,\cdots, L,\\
        & \sum_{j\in\J}\Dv\,\chi_{k,j}=1,\\
        & 0 <\,\underline{D}_k\,\leq D_k^\Delta\leq\,\overline{D}_k\,<\infty,\qquad Q_k^\Delta\leq\,\overline{Q}_k\,<\infty,
    \end{aligned}
\end{equation}
where 
\[ D_k^\Delta\coloneq\sum_{j\in\J}\Dv\,|v_j|^2\,\chi_{k,j},\qquad Q_k^\Delta\coloneq\sum_{j\in\J}\Dv\,|v_j|^4\,\chi_{k,j}=\sum_{j\in\J}\Dv\,v_j^2|v_j|^2\,\chi_{k,j},\]
and $\underline{D}_k$, $\overline{D}_k$, $\overline{Q}_k$ are universal constants. Typically, we define $\chi_{k,j}=c_{\Dv}\chi_k(v_j)$ and compute $c_{\Dv}$ in such a way that the mass of $(\chi_{k,j})_j$ is 1. Note also that the symmetry properties imply 
\begin{equation}\label{sym_chi_dis}
    \sum_{j\in\J}\Dv\,v_j\,\chi_{k,j}=0.
\end{equation}

\subsubsection{Discrete functional setting}
To reproduce at the discrete level the functional framework used in the continuous analysis, we now introduce discrete counterparts of the weighted $L^2$ scalar products and of the associated norms for both microscopic and macroscopic quantities.

For discrete microscopic quantities $\Fbf=\left(f_{1,ij},f_{2,ij}\right)^\top_{ij}$ and $ \Gbf=\left( g_{1,ij}, g_{2,ij} \right)^\top_{ij}$, we define the discrete weighted $L^2$ scalar product by
\begin{equation}\label{def_psL2weighted_sdis}
    \langle \Fbf,\Gbf\rangle_\Delta\coloneq\sum_{(i,j)\in\II\times\J}\Dx\Dv\,\left(\frac{f_{1,ij}g_{1,ij}}{f_{1,j}^{\infty,*}} + \frac{f_{2,ij}g_{2,ij}}{f_{2,j}^{\infty,*}}\right),
\end{equation}
and denote by $\|\cdot\|_\Delta$ the associated norm. For macroscopic quantities $\mathbf{U}=\left(u_{1,i}, u_{2,i} \right)^\top_{i\in\II}$ and $\mathbf{W}=\left( w_{1,i}, w_{2,i} \right)^\top_{i\in\II}$, we also define a discrete weighted scalar product in position $\lla \cdot,\cdot\rra_{\Delta,x}$ as
\begin{equation}\label{def_psL2weightedposition_sdis}
    \lla \mathbf{U}, \mathbf{W}\rra_{\Delta,x} = \sum_{i\in\II}\Dx\left(\frac{u_{1,i}w_{1,i}}{\rho_1^{\infty,*}} + \frac{u_{2,i}w_{2,i}}{\rho_2^{\infty,*}}\right)
\end{equation}
with $\|\cdot\|_{\Delta,x}$ the associated norm. 

Due to our choices of discretization, we need to define several discrete gradients in position. Given a macroscopic quantity $u=(u_i)_{i\in\II}$, we define 
\begin{itemize}
\item the discrete centered gradient $D_x^c u\in\R^\II$ given by
\[(D_x^cu)_i=\frac{u_{i+1}-u_{i-1}}{2\Dx} \qquad \forall i\in\II,\]
\item the discrete downstream gradient $D_x^-u\in\R^\II$ given by
\[(D_x^-u)_i=\frac{u_{i}-u_{i-1}}{\Dx} \qquad \forall i\in\II,\]
\item the discrete upstream gradient $D_x^+u\in\R^\II$ given by
\[(D_x^+u)_i=\frac{u_{i+1}-u_{i}}{\Dx} \qquad \forall i\in\II.\]
\end{itemize}
It is straightforward to see that these discrete gradients satisfy the following properties:
\begin{equation}\label{prop_gradients_dis}
    \frac{1}{2}(D_x^-+D_x^+)=D_x^c, \qquad D_x^+D_x^-=D_x^-D_x^+.
\end{equation}

By convention, we assume that these discrete differentiation operators act componentwise on vector-valued quantities. In particular, for a discrete quantity $\mathbf{U}=(U_1,U_2)^\top$, we set
\[D_x^*\mathbf{U}=(D_x^*U_1,D_x^*U_2)^\top, \quad \text{for }*\in\{c,+,-\}.\]
Using the definition of the discrete gradients and the periodic boundary conditions, we immediately have the following algebraic properties.

\begin{Lemma}\label{ref_IPP_discret}
   For all discrete macroscopic quantities $\mathbf{U}$ and $\mathbf{W}$, it holds
    \begin{align}
       & \langle D_x^c \mathbf{U},\mathbf{W}\rangle_{\Delta,x}=-\langle \mathbf{U},D_x^c\mathbf{W}\rangle_{\Delta,x}, \label{IPP_centre}\\
       & \langle D_x^+ \mathbf{U},\mathbf{W}\rangle_{\Delta,x}=-\langle \mathbf{U},D_x^-\mathbf{W}\rangle_{\Delta,x} \label{IPP_decentre},\\
       & \langle (D_x^+D_x^- + D_x^-D_x^+) \mathbf{U},\mathbf{W}\rangle_{\Delta,x}=-4\langle D_x^c\mathbf{U},D_x^c\mathbf{W}\rangle_{\Delta,x} \label{IPP2_decentre},\\
       & \Dx\|D_x^c \mathbf{U}\|_{\Delta,x} \leq \|\mathbf{U}\|_{\Delta,x} \label{estim_dxu_disc}.
    \end{align}
\end{Lemma}

Let us also recall the discrete Poincaré inequality on the torus (see for example \cite[Lemma 6]{BessemoulinHerdaRey2020} for a proof of this result).

\begin{Lemma}[Discrete Poincaré inequality on the torus]\label{lem_Poincare_dis}
    Assume that the number of points $N$ used in the spatial discretization of the torus is odd. Then, there is a constant $C_P>0$ independent of $\Dx$ such that for all $\mathbf{U}=\left( u_{1,i}, u_{2,i} \right)^\top_{i\in\II}$ satisfying $\sum_{i\in\II}\Dx\,u_{k,i}=0$, $k=1,2$,
    \[\|\mathbf{U}\|_{\Delta,x}\leq C_P\,\|D_x^c \mathbf{U}\|_{\Delta,x}.\]
\end{Lemma}
Note that in the above lemma, the constant $C_P$ approaches $\frac{1}{\pi}$ whereas in the continuous case the Poincaré constant is $\frac{1}{2\pi}$. This discrepancy arises from the specific choice of the centered discrete gradient, and we refer the reader to \cite{BessemoulinHerdaRey2020} for a detailed discussion.

\subsection{Definition of the numerical scheme}
The numerical scheme used to approximate \eqref{eq_1_nonlin}–\eqref{eq_2_nonlin} is based on a fully implicit finite-volume discretization in the phase space. More precisely, at each time step $t^n = n\Dt$ and for each control volume $K_{ij}$, the solution is approximated by $\Fbf^n_{ij}=(f^n_{1,ij},f^n_{2,ij})^\top$. The initial datum $\Fbf^\tin=(f_{1}^\tin,f_{2}^\tin)^\top$ is discretized as
\[
f_{k,ij}^\tin=\frac{1}{\Dx\,\Dv}\int_{K_{ij}}f_k^\tin(x,v)\,\dd x\,\dd v, \quad \forall (i,j)\in\II\times\J,\quad k =1,\,2.
\]
Integrating \eqref{eq_1_nonlin}--\eqref{eq_2_nonlin} over each cell $K_{ij}$ and applying a fully implicit time discretization leads to the following scheme: for all $n\geq 0$, $i\in\II$, $j\in\J$,
\begin{align}
    & \frac{f_{1,ij}^{n+1}-f_{1,ij}^n}{\Delta t}+\frac{1}{\Dx\,\Dv}\left(\mathcal{F}_{\iph,j}^{n+1}-\mathcal{F}_{\imh,j}^{n+1}\right)=\chi_{1,j}-\rho_{2,i}^{n+1}f_{1,ij}^{n+1} ,\label{scheme_1_nonlin_d1}\\
    &  \frac{f_{2,ij}^{n+1}-f_{2,ij}^n}{\Delta t}+\frac{1}{\Dx\,\Dv}\left(\mathcal{G}_{\iph,j}^{n+1}-\mathcal{G}_{\imh,j}^{n+1}\right)=\chi_{2,j}-\rho_{1,i}^{n+1}f_{2,ij}^{n+1},\label{scheme_2_nonlin_d1}
\end{align}
with the discrete densities
\begin{equation}\label{def_density_dis}
\rho_{k,i}^{n+1} = \sum_{j \in \J}\Delta v\, f_{k,ij}^{,n+1} , \qquad k=1,2.
\end{equation}
We consider upwind fluxes, which are defined by
\begin{align}
    & \mathcal{F}^n_{\iph,j}=\Dv\left(v_j^+ \,f^n_{1,ij} + v_j^- \,f^n_{1,i+1,j}\right),\label{flux_F_UW_sd1}\\
    & \mathcal{G}^n_{\iph,j}=\Dv \left(v_j^+ \,f^n_{2,ij} + v_j^-\,f^n_{2,i+1,j}\right) \label{flux_G_UW_sd1},
\end{align}
where
\begin{equation*}
    \begin{aligned}
        a^+ &= \max (a,0)=(|a|+a)/2,\\
        a^- &= \min (a,0)=(|a|-a)/2.
    \end{aligned}
\end{equation*} 

Remark that the scheme \eqref{scheme_1_nonlin_d1}--\eqref{scheme_2_nonlin_d1} clearly satisfies the discrete mass conservation of the difference $f_1-f_2$:
\[\sum_{(i,j)\in\II\times\J}\Dx\,\Dv\,(f^n_{1,ij}-f^n_{2,ij})=\sum_{(i,j)\in\II\times\J}\Dx\,\Dv\,(f_{1,ij}^\tin-f_{2,ij}^\tin) \quad \forall n\geq 0.\]
Then, as in the continuous setting, we define $\rhoinfs>0$ as the unique constant such that 
\begin{equation}\label{def_rhoinf_sd}
    M_0\coloneq\sum_{(i,j)\in\II\times\J}\Dx\,\Dv\,(f_{1,ij}^\tin-f_{2,ij}^\tin)=|\T|\left(\rhoinfs-\frac{1}{\rhoinfs}\right).
\end{equation}
In particular, we take
\begin{equation}\label{def_rhoinf_sdis_full}
    \rhoinfs = \frac{M_0+\sqrt{M_0^2+4|\T|}}{2|\T|}.
\end{equation}
It is clear that $\Fbf_{ij}^{\infty,*}=(f_{1,ij}^{\infty,*},f_{2,ij}^{\infty,*})^\top$ defined by
\begin{equation}\label{def_eq_sd}
    f^{\infty,*}_{1,ij}=\rho_1^{\infty,*}\,\chi_{1,j},\qquad f^{\infty,*}_{2,ij}=\rho_2^{\infty,*}\,\chi_{2,j} \qquad \forall (i,j)\in \II\times\J,
\end{equation}
with $\rho_1^{\infty,*}=\rhoinfs$ and $\rho_2^{\infty,*}=1/\rhoinfs$, is an equilibrium for the scheme \eqref{scheme_1_nonlin_d1}--\eqref{scheme_2_nonlin_d1}.

As in \cite[Theorem 3]{BCLR2025}, it is possible to establish the existence of a solution to the numerical scheme \eqref{scheme_1_nonlin_d1}--\eqref{scheme_2_nonlin_d1}, as well as a discrete maximum principle.
\begin{Theorem}\label{thm_bounds_existence}
    Assume that the discrete velocity profiles satisfy \eqref{hyp_chi_dis}, and that there exist constants $0<\rho_m\leq \rho_M<\infty$ such that for all $i\in\II$, $j\in\J$
    \begin{align}
        &\rho_m\chi_{1,j}\leq f^{\tin}_{1,ij}\leq \rho_M\chi_{1,j}, \label{hyp.CI.1_dis}\\
        &\rho_M^{-1}\chi_{2,j}\leq f^{\tin}_{2,ij}\leq \rho_m^{-1}\chi_{2,j}. \label{hyp.CI.2_dis}
    \end{align}
    Then the scheme \eqref{scheme_1_nonlin_d1}--\eqref{scheme_2_nonlin_d1} admits a solution $(f_{1,ij}^n,f_{2,ij}^n)_{i\in\II,j\in\J,n\geq 0}$ satisfying for all $i\in\II$, $j\in\J$ and $n\geq 0$,
   \begin{align}
        &\rho_m\chi_{1,j}\leq f^{n}_{1,ij}\leq \rho_M\chi_{1,j}, \label{estimLinf_f1_d}\\
        &\rho_M^{-1}\chi_{2,j}\leq f^{n}_{2,ij}\leq \rho_m^{-1}\chi_{2,j}. \label{estimLinf_f2_d}
    \end{align} 
\end{Theorem}

The proof follows the approach in \cite{BCLR2025}, with the Lax-Friedrichs fluxes replaced by the upwind ones. It proceeds in two steps: first, existence is shown for a truncated version of the scheme; second,  one proves that any solution to this truncated scheme automatically satisfies the $L^\infty$ bounds \eqref{estimLinf_f1_d}--\eqref{estimLinf_f2_d}, and therefore also solves the original scheme. This result relies on the monotonicity of the fluxes, a property verified by the upwind discretization \eqref{flux_1_UW_d}--\eqref{flux_2_UW_d}. A detailed proof is provided in Appendix~\ref{app:ExistenceMaxPrinciple}.\\

Now, since we want to study the convergence of $(\Fbf^n_{ij})_{n,i,j}$ to $(\Fbf^{\infty,*}_{ij})_{i,j}$, we introduce the deviation $\Ftbf^n_{ij}\coloneq \Fbf^n_{ij}-\Fbf_{ij}^{\infty,*}$ and rewrite the scheme \eqref{scheme_1_nonlin_d1}--\eqref{scheme_2_nonlin_d1} as
\begin{align}
    & \frac{\ft_{1,ij}^{n+1}-\ft_{1,ij}^n}{\Delta t}+\frac{1}{\Dx\,\Dv}\left(\fluxFt_{\iph,j}^{n+1}-\fluxFt_{\imh,j}^{n+1}\right)=\chi_{1,j}-\rho_{2,i}^{n+1}f_{1,ij}^{n+1} ,\label{scheme_1_nonlin_d}\\
    &  \frac{\ft_{2,ij}^{n+1}-\ft_{2,ij}^n}{\Delta t}+\frac{1}{\Dx\,\Dv}\left(\fluxGt_{\iph,j}^{n+1}-\fluxGt_{\imh,j}^{n+1}\right)=\chi_{2,j}-\rho_{1,i}^{n+1}f_{2,ij}^{n+1},\label{scheme_2_nonlin_d}
\end{align}
where the numerical fluxes $\fluxFt_{\iph,j}$ and $\fluxGt_{\iph,j}$ are defined by \eqref{flux_F_UW_sd1} and \eqref{flux_G_UW_sd1}, replacing $f_k$ by~$\ft_k$.

In \cite{BCLR2025}, the discrete hypocoercivity of the linearized system associated with \eqref{eq_1_nonlin}–\eqref{eq_2_nonlin}, discretized using Lax–Friedrichs fluxes, was studied. The central-flux-plus-viscosity formulation of these fluxes proved particularly convenient, as it allowed one to clearly distinguish the discrete analogues of the continuous terms from the additional contributions arising solely from numerical viscosity. Following this formalism, the upwind fluxes can be equivalently written as
\begin{align}
    & \fluxFt^n_{\iph,j}=\Dv\left(v_j\frac{\ft^n_{1,ij}+\ft_{1,i+1,j}}{2}-\frac{|v_j|}{2}(\ft^n_{1,i+1,j}-\ft^n_{1,ij})\right),\label{flux_1_UW_d}\\
    & \fluxGt^n_{\iph,j}=\Dv\left(v_j\frac{\ft^n_{2,ij}+\ft_{2,i+1,j}}{2}-\frac{|v_j|}{2}(\ft^n_{2,i+1,j}-\ft^n_{2,ij})\right) \label{flux_2_UW_d}.
\end{align}
This representation highlights the decomposition into a centered term and a stabilizing numerical diffusion, in direct analogy with the Lax–Friedrichs formulation.

\subsection{Definition and estimates of the discrete moments}

We now introduce the discrete moments, derive their evolution equations, and establish the corresponding estimates.

For $n\geq0$, $i\in\II$ and $k=1,2$, let
\begin{equation}\label{def_moments_sdiscrets}
   J^n_{k,i}\coloneq\sum_{j\in\J}\Dv\,v_j\,f^n_{k,ij}, \quad S^n_{k,i}\coloneq\sum_{j\in\J}\Dv\,(v_j^2-D_k^\Delta)\,f^n_{k,ij},
\end{equation}
and remember that $\rho^n_{k,i}$ is defined by \eqref{def_density_dis}.

We further define the following ``skewed'' discrete moments, which will appear in the discrete moments equations as viscous terms. For $n\geq0$, $i\in\II$ and $k=1,2$, 
\begin{equation}\label{def_moments_sdiscrets_skewed}
    J^{\mathrm{s},n}_{k,i}\coloneq\sum_{j\in\J}\Dv\,|v_j|\,f^n_{k,ij}, \quad S^{\mathrm{s},n}_{k,i}\coloneq\sum_{j\in\J}\Dv\,v_j|v_j|\,f^n_{k,ij}.
\end{equation}
In the following, we use the vectorial notations:
\begin{equation}\label{def_moments_sdiscrets_vect}
    \rhobf^n_{\Fbf,i} =\begin{pmatrix} \rho^n_{1,i}\\\rho^n_{2,i} \end{pmatrix},\, \Jbf^n_{\Fbf,i}=\begin{pmatrix} J^n_{1,i}\\J^n_{2,i} \end{pmatrix},\, \Sbf^n_{\Fbf,i}=\begin{pmatrix} S^n_{1,i}\\S^n_{2,i} \end{pmatrix},\, \Jbf^{\mathrm{s},n}_{\Fbf,i}=\begin{pmatrix} J^{\mathrm{s,n}}_{1,i}\\J^{\mathrm{s,n}}_{2,i} \end{pmatrix},\, \Sbf^{\mathrm{s},n}_{\Fbf,i}=\begin{pmatrix} S^{\mathrm{s},n}_{1,i}\\S^{\mathrm{s},n}_{2,i} \end{pmatrix}.
\end{equation}
To lighten the notations, the moments of the deviation $\Ftbf=\Fbf-\Fbf^\infty$ are denoted by
\begin{equation*}
    \rhobf^n_{\Ftbf,i}\coloneq\rhotbf^{\,n}_i,\quad \Jbf^n_{\Ftbf,i}\coloneq\Jtbf^n_i,\quad \Sbf^n_{\Ftbf,i}\coloneq\Stbf^n_i,\quad \Jbf^{\mathrm{s},n}_{\Ftbf,i}\coloneq\Jtbf^{\mathrm{s},n}_i,\quad \Sbf^{\mathrm{s},n}_{\Ftbf,i}\coloneq\Stbf^{\mathrm{s},n}_i.
\end{equation*}
Let us now define the discrete counterpart of the projection $\Pi$ as
\begin{equation}\label{def_Pi_discret}
    (\Pi^\Delta \Fbf)_{ij}\coloneq \begin{pmatrix}\rho_{1,i}\,\chi_{1,j}\\\rho_{2,i}\,\chi_{2,j}\end{pmatrix}.
\end{equation}

Using the equivalent formulation \eqref{flux_1_UW_d}–\eqref{flux_2_UW_d} of the upwind fluxes, the discrete moment equations are easily derived by multiplying the scheme \eqref{scheme_1_nonlin_d}--\eqref{scheme_2_nonlin_d} by $(1, v_j)^\top$ and summing over the velocity index~$j$.

\begin{Lemma}[Moments schemes]\label{lem_eq_moments_dis}
The discrete moments satisfy the following schemes: for all $n\geq 0$, $i\in\II$, 
    \begin{align}
        & \frac{ \rhotbf_{i}^{n+1}-\rhotbf_{i}^n}{\Delta t}+(D_x^c \Jtbf^{n+1})_i-\frac{\Dx}{2}\left((D_x^+D_x^-+D_x^-D_x^+)\Jtbf^{\,\mathrm{s},n+1}\right)_i=(1-\rho_{1,i}^{n+1}\rho_{2,i}^{n+1})\begin{pmatrix}1\\1 \end{pmatrix}, \label{scheme_rho_UW_d}\\
        &  \frac{\Jt_{1,i}^{n+1}-\Jt_{1,i}^n}{\Delta t}+(D_x^c\St^{n+1}_1)_i+D_1^\Delta\,(D_x^c\rhot^{\,n+1}_1)_i\nonumber \\
        & \qquad\qquad -\frac{\Dx}{2}\left((D_x^+D_x^-+D_x^-D_x^+)\St^{\,\mathrm{s},n+1}_1\right)_i=-\rhot^{\,n+1}_{2,i}\Jt^{n+1}_{1,i},\label{scheme_J1_UW_d}\\
        &  \frac{\Jt^{n+1}_{2,i}-\Jt^{n+1}_{2,i}}{\Delta t}+(D_x^c\St^{n+1}_2)_i+D_2^\Delta\,(D_x^c\rhot^{\,n+1}_2)_i\nonumber \\
        & \qquad\qquad-\frac{\Dx}{2}\left((D_x^+D_x^-+D_x^-D_x^+)\St^{\,\mathrm{s},n+1}_2\right)_i= -\rhot^{\,n+1}_{1,i}\Jt^{n+1}_{2,i}.\label{scheme_J2_UW_d}
    \end{align}
\end{Lemma}


Similarly as in the continuous setting, the discrete moments of the deviation $\Ftbf=\Fbf-\Fbf^\infty$ can be estimated using purely algebraic properties.
\begin{Lemma}[Discrete moments estimates]\label{lem_mom_estim_disc}
    The discrete moments satisfy the following estimates for all $t\geq 0$:
    \begin{align}
        \|\rhotbf^{n}\|_{\Delta,x}&=\|\Pi^\Delta \Ftbf^n\|_\Delta, \label{estim_rhot_dis} \\
        \|\Jtbf^n\|_{\Delta,x}&\leq C_{1}^*\,\|(I-\Pi^\Delta)\Ftbf^n\|_\Delta,\label{estim_J2_dis}\\
        \|\Stbf^{n}\|_{\Delta,x}&\leq C_{2}^*\,\|(I-\Pi^\Delta)\Ftbf^n\|_\Delta, \label{estim_S_dis}\\
        \|\Jtbf^{\mathrm{s},n}\|_{\Delta,x}&\leq C_1^*\,\|\Ftbf^n\|_\Delta, \label{estim_Jb_dis}\\ 
        \|\Stbf^{\mathrm{s},n}\|_{\Delta,x}&\leq C_{3}^*\,\|(I-\Pi^\Delta)\Ftbf^n\|_\Delta, \label{estim_Sb_dis}
    \end{align}
    where the constants are given by:
    \begin{align}
        C_{1}^* &= \sqrt{\max\lbr\overbar{D_1},\overbar{D_2}\rbr},\label{def.CJ*}\\
        C_{2}^* &= \sqrt{\max\lbr (\overbar{Q_1}-\underline{D_1}^2,\,\overbar{Q_2}-\underline{D_2}^2 \rbr},\\
        C_3^* &= \sqrt{\max\lbr \overbar{Q_1},\,\overbar{Q_2} \rbr}.
    \end{align}
\end{Lemma}
The proof of this lemma is obtained with the same algebraic manipulations as in the proof of Lemma~\ref{lem_mom_estim}, relying on the symmetry properties of the velocity mesh and of the discrete equilibrium profiles \eqref{hyp_chi_dis}--\eqref{sym_chi_dis}. The only additional step consists in estimating the new skewed moments, which can be handled using similar arguments.

\begin{proof}
    We obtain estimate \eqref{estim_Jb_dis} thanks to definition \eqref{def_moments_sdiscrets_skewed} and the Cauchy-Schwarz inequality: 
    \begin{align*}
    \|\Jtbf^{\mathrm{s},n}\|_{\Delta,x}^2&=\sum_{k=1,2}\frac{1}{\rho_k^{\infty,*}}\sum_{i\in\II}\Dx\left(\sum_{j\in\J}\Dv\,|v_j|\ft^n_{k,ij}\frac{\sqrt{\chi_{k,j}}}{\sqrt{\chi_{k,j}}}\right)^2\\
    &\leq \sum_{k=1,2}\frac{1}{\rho_k^{\infty,*}}\sum_{i\in\II}\Dx\left(\sum_{j\in\J}\Dv\, v_j^2\chi_{k,j}\right)\left(\sum_{j\in\J}\Dv\frac{\left(\ft^n_{k,ij}\right)^2}{\chi_{k,j}}\right)\\
    &\leq \max(\overline{D_1},\overline{D_2})\,\|\Ftbf^n\|_{\Delta}^2.
    \end{align*}

    Estimate \eqref{estim_Sb_dis} is a consequence of the following relation (due to symmetry): 
    \begin{equation*}
        \sum_{j\in\J}\Dv\, v_j|v_j|\chi_{k,j}=0.
    \end{equation*}
    Indeed, one has:
    \begin{equation}
        \begin{aligned}
            \|\Stbf^{\mathrm{s},n}\|_{\Delta,x}^2  &= \sum_{k=1,2}\sum_{i\in\II} \Dx\frac{1}{\rho^{\infty,*}_k}\left(\sum_{j\in\J}\Dv\, v_j|v_j|\ft^n_{k,ij}\right)^2\\
            &= \sum_{k=1,2}\sum_{i\in\II} \Dx\frac{1}{\rho^{\infty,*}_k}\left(\sum_{j\in\J}\Dv\, v_j|v_j|\left(\ft^n_{k,ij}-\rhot_{k,i}\chi_{k,j}\right)\right)^2\\
            &\leq \max\lbr \overline{Q_1}, \overline{Q_2} \rbr \|(I-\Pi^\Delta)\Ftbf^n\|_\Delta^2.
        \end{aligned}
    \end{equation}
\end{proof}

\section{Discrete nonlinear hypocoercivity}\label{sec:Hypoco}
In this section, we adapt the proof of Section~\ref{sec:ContinuousSetting} to the discrete setting. Since algebraic manipulations directly carry over to the discrete setting, we shall focus on the main challenges induced by the discretization.

The discrete analog of the relative Boltzmann entropy is defined as 
\begin{equation}\label{BoltzEntropy_d}
    H^n\coloneq \sum_{k=1,2}\sum_{i,j}\Dx\Dv\left(f_{k,ij}^n\left(\log\left(\frac{f_{k,ij}^n}{f_{k,j}^{\infty,*}}\right)-1\right)+f_{k,j}^{\infty,*}\right).
\end{equation}
First, we introduce the discrete counterpart of Lemma \ref{lem_microcoercivity}.

\begin{Lemma}[Discrete microscopic coercivity]\label{lem_disc_microcoercivity}
    Under assumptions \eqref{hyp_chi_dis} and \eqref{hyp.CI.1_dis}--\eqref{hyp.CI.2_dis}, let $\Fbf=(\Fbf_{ij})_{ij}$ be solution to the scheme \eqref{scheme_1_nonlin_d1}--\eqref{scheme_2_nonlin_d1}. Then for every $n\geq 0$, 
    \begin{equation}\label{estim_coercivite_micro_dis}
        \frac{H^{n+1} - H^n}{\Delta t} \leq -  \mathcal{D}^{n+1}-\frac{C_{4}^*}{\Delta t}\|\Ftbf^{n+1}-\Ftbf^n\|_{\Delta}^2,
    \end{equation}
    where
    \begin{equation}\label{def.D_d}
        \mathcal{D}^n = \sum_{(i,j,m)\in\II\times\J^2}\Dx\Dv^2 \left(\chi_{1,j}\chi_{2,m} - f_{1,ij}^nf_{2,im}^n\right)\log\left(\frac{f_{1,ij}^nf_{2,im}^n}{\chi_{1,j}\chi_{2,m}}\right),
    \end{equation}
    and 
    \[C_{4}^*=\frac{1}{2}\frac{\rho_m}{\rho_M}\min\left(\frac{\rho_1^{\infty,*}}{\rho_M},\rho_2^{\infty,*}\rho_m\right).\]
\end{Lemma}
For the purely algebraic manipulations carried out in the continuous setting, the transposition to the discrete level is straightforward. The main difference arises from the absence of a discrete chain rule, which forces us to rely instead on the concavity of the logarithm function. More precisely, the remainder term in the first-order Taylor expansion of $s \mapsto \log(s)$ has a definite sign, which we exploit to obtain the desired estimate. This yields the second term in \eqref{estim_coercivite_micro_dis}, corresponding to an additional numerical dissipation. Furthermore, this extra dissipation originates purely from the time discretization and is of order $\Dt$; hence, it vanishes in the continuous-time limit $\Dt \to 0$.
\begin{proof}
    By definition, we have
    \begin{equation*}
        H^{n+1}-H^n=\sum_{k=1,2}\sum_{i,j}\Dx \Dv \left[f_{k,ij}^{n+1}\left(\log\left(\frac{f_{k,ij}^{n+1}}{f_{k,j}^\infty}\right)-1\right)-f_{k,ij}^{n}\left(\log\left(\frac{f_{k,ij}^{n}}{f_{k,j}^\infty}\right)-1\right)\right].
    \end{equation*}
    For $k=1,2$, for all $i\in\II$, $j\in\J$, we can write
    \begin{gather*}
        f_{k,ij}^{n+1}\left(\log\left(\frac{f_{k,ij}^{n+1}}{f_{k,j}^\infty}\right)-1\right)-f_{k,ij}^{n}\left(\log\left(\frac{f_{k,ij}^{n}}{f_{k,j}^\infty}\right)-1\right) = \\
        (f_{k,ij}^{n+1}-f_{k,ij}^n)\log\left(\frac{f_{k,ij}^{n+1}}{f_{k,j}^\infty}\right)+f_{k,ij}^n\left[\log\left(\frac{f_{k,ij}^{n+1}}{f_{k,j}^\infty}\right)-\log\left(\frac{f_{k,ij}^n}{f_{k,j}^\infty}\right)\right]+f_{k,ij}^{n+1}-f_{k,ij}^n.
    \end{gather*}
    Using a Taylor expansion, the second term of the right hand side can be written as
    \[ 
        f_{k,ij}^n\left(\log(f_{k,ij}^{n+1})-\log(f_{k,ij}^n)\right) = f_{k,ij}^n\left[\frac{1}{f_{k,ij}^n}(f_{k,ij}^{n+1}-f_{k,ij}^n)-\frac{1}{2(\widehat{f_{k,ij}})^2}(f_{k,ij}^{n+1}-f_{k,ij}^n)^2\right],
    \]
    with $\widehat{f_{k,ij}}\in\left (\min(f_{k,ij}^n,f_{k,ij}^{n+1}),\max(f_{k,ij}^n,f_{k,ij}^{n+1})\right)$.
    We thus have
    \begin{gather}
        H^{n+1}-H^n=\sum_{k=1,2}\sum_{i,j}\Dx \Dv\, (f_{k,ij}^{n+1}-f_{k,ij}^n)\log\left(\frac{f_{k,ij}^{n+1}}{f_{k,j}^\infty}\right) \nonumber\\
        -\frac{1}{2}\sum_{k=1,2}\sum_{i,j}\Dx \Dv \,\frac{f_{k,ij}^n}{(\widehat{f_{k,ij}})^2}\,(f_{k,ij}^{n+1}-f_{k,ij}^n)^2.\label{micro_d1}
    \end{gather}
    The first term is analog to its continuous counterpart, and is rewritten using the numerical scheme \eqref{scheme_1_nonlin_d1}--\eqref{scheme_2_nonlin_d1}, yielding
    \begin{align}
        &\sum_{k=1,2}\sum_{i,j}\Dx\Dv(f_{k,ij}^{n+1}-f_{k,ij}^n)\log\left(\frac{f^{n+1}_{k,ij}}{f_{k,j}^{\infty,*}}\right)\label{micro_d2}\\
        &=\Dt\sum_{i,j}\left(-(\mathcal{F}^{n+1}_{\iph}-\mathcal{F}^{n+1}_{\imh})+\Dx\Dv(\chi_{1,j}-\rho^{n+1}_{2,i}f^{n+1}_{1,ij})\right)\log\left(\frac{f^{n+1}_{1,ij}}{f_{1,j}^{\infty,*}}\right)\nonumber\\
        &+ \Dt\sum_{i,j}\left(-(\mathcal{G}^{n+1}_{\iph}-\mathcal{G}^{n+1}_{\imh})+\Dx\Dv(\chi_{2,j}-\rho_{1,i}f^{n+1}_{2,ij})\right)\log\left(\frac{f^{n+1}_{2,ij}}{f_{2,j}^{\infty,*}}\right).\nonumber
    \end{align}
    The transport terms are handled in the same way for both species. We therefore focus on the first one. Using a discrete integration by parts on $i$, one has 
    \begin{equation*}
    \begin{aligned}
     A_\TT &\coloneq  -\sum_{i,j}(\mathcal{F}^{n+1}_{\iph,j}-\mathcal{F}^{n+1}_{\imh,j})\log\left(\frac{f^{n+1}_{1,ij}}{f_{1,j}^{\infty,*}}\right) \\
     &= \sum_{i,j} \mathcal{F}^{n+1}_{\iph,j}\left(\log\left(f^{n+1}_{1,i+1,j}\right)-\log\left(f^{n+1}_{1,ij}\right)\right)\\
        &= \sum_{i\in\II, j>0}\Dv\, v_j f_{1,ij}\left(\log\left(f^{n+1}_{1,i+1,j}\right)-\log\left(f^{n+1}_{1,ij}\right)\right)\\
        &\quad+ \sum_{i\in\II, j<0}\Dv\, v_j f^{n+1}_{1,i+1,j} \left(\log\left(f^{n+1}_{1,i+1,j}\right)-\log\left(f_{1,ij}\right)\right).
    \end{aligned}
    \end{equation*}
    Then, by concavity of $s\mapsto\log(s)$, we have
    \begin{equation*}
        \begin{aligned}
            A_\TT       &\leq \sum_{i\in\II, j>0}\Dv\, v_j f_{1,ij}^{n+1}\frac{f_{1,i+1,j}^{n+1}-f_{1,ij}^{n+1}}{f_{1,ij}^{n+1}}
            + \sum_{i\in\II, j<0}\Dv\, v_j f_{1,i+1,j}^{n+1} \frac{f_{1,i+1,j}^{n+1}-f_{1,ij}^{n+1}}{f_{1,i+1,j}^{n+1}}\\
            &\leq 0,
        \end{aligned}
    \end{equation*}
    where we have used the periodic boundary conditions in the physical space to conclude. The reaction terms in \eqref{micro_d2} are then handled in the same way as in the continuous setting, since their treatment involves only algebraic manipulations. One then obtains
    \begin{equation*}
    \sum_{k=1,2}\sum_{i,j}\Dx \Dv\, (f_{k,ij}^{n+1}-f_{k,ij}^n)\log\left(\frac{f_{k,ij}^{n+1}}{f_{k,j}^\infty}\right) \leq -\Delta t \mathcal{D}^{n+1}.
    \end{equation*}
    For the second term of the right-hand side of \eqref{micro_d1}, we use the uniform-in-time $L^\infty$ bounds of $f_{k,ij}^n$. These bounds immediately transfer to $\widehat{f_{k,ij}}$, from which we obtain
    \[\frac{f_{1,ij}^n}{(f_{1,ij}^*)^2}\geq \frac{\rho_m\chi_1}{(\rho_M\chi_1)^2}=\frac{\rho_m}{\rho_M^2}\frac{1}{\chi_1}, \quad \frac{f_{2,ij}^n}{(f_{2,ij}^*)^2}\geq \frac{\chi_2\rho_m^2}{\rho_M\chi_2^2}=\frac{\rho_m^2}{\rho_M}\frac{1}{\chi_2},\]
    and finally the expected result \eqref{estim_coercivite_micro_dis}.
\end{proof}

By directly adapting the arguments used in the continuous setting to the discrete framework, we obtain the following discrete analogue of Lemma \ref{lem_estimD}. It relates the discrete dissipation to the nonlinear reaction terms appearing in the discrete moment equation \eqref{scheme_rho_UW_d}, as well as to the dissipation in the velocity variable.
\begin{Lemma}\label{lem_estimD_d}
Under assumptions \eqref{hyp_chi_dis} and \eqref{hyp.CI.1_dis}--\eqref{hyp.CI.2_dis}, the dissipation term defined by \eqref{def.D_d} satisfies for all $n\geq 0$
\begin{equation}\label{estim.D_d}
    \mathcal{D}^n\geq C_5^*\|1-\rho_1^n\rho_2^n\|_{L^2(\T)}^2+C_6^*\|(I-\Pi^\Delta)\Ftbf^n\|^2_\Delta,
\end{equation}
where $C_5^*=\rho_m/\rho_M$ and $C_6^*=(\rho_m/\rho_M)(\rho_1^{\infty,*}+\rho_2^{\infty,*})^{-1}\min(\rho_m,\rho_M^{-1})^2$.
\end{Lemma}

We are now ready to state the full discrete hypocoercivity result.

\begin{Theorem}\label{thm.hypoco_d}
    Under assumptions \eqref{hyp_chi_dis} and \eqref{hyp.CI.1_dis}--\eqref{hyp.CI.2_dis}, and assuming that the number of points $N$ of the spatial discretization is odd, there exist constants $C\geq 1$ and $\kappa>0$, which do not depend on the discretization parameters, such that 
    \begin{equation}
    \|\Fbf^n-\Fbf^{\infty,*}\|_\Delta\leq C\,\|\Fbf^\tin-\Fbf^{\infty,*}\|_\Delta e^{-\kappa t^n} \quad \forall n\geq 0.
    \end{equation}
\end{Theorem}

In order to prove this result, we introduce the discrete modified entropy functional
\begin{equation}\label{eq_GammaDelta}
    \Gamma^n\coloneq  H^n+\delta\lla \Jtbf^n,D_x^c\Phibf^n\rra_{\Delta,x}+\frac{\delta}{2\Delta t}\sum_{i\in\II}\Dx \left((D_x^c\Phibf^n)_i-(D_x^c\Phibf^{n-1})_i\right)^2,
\end{equation}
where $\Phibf^n=\left(\phi^n_1, \phi^n_2\right)^\top$, $(\phi_{k,i})_{i\in\II}$ being solution to the discrete Poisson equation
\begin{equation}\label{Poisson_d}
    -D_x^cD_x^c\phi^n_{k,i}=\rhot_{k,i},\quad\sum_{i\in\II}\Dx\phi^n_{k,i}=0, \quad k=1,2,
\end{equation}
and $\delta>0$ is a small parameter to be chosen later. For an odd number of points $N$, existence and uniqueness of $\Phibf^n$ satisfying \eqref{Poisson_d} is guaranteed (see \cite{BessemoulinHerdaRey2020}). The third term in \eqref{eq_GammaDelta} arises from the choice of a fully implicit time discretization and is consistent with zero in the limit $\Dt \to 0$. This term, similarly to the additional numerical dissipation identified in Lemma~\ref{lem_disc_microcoercivity}, serves to compensate for the lack of a discrete chain rule, which necessitates to control the corresponding remainder terms.\\

Analogously to the continuous setting, it is straightforward to show that the discrete modified entropy $\Gamma^n$ defines a norm equivalent to $\|\Ftbf^n\|^2_\Delta$, as stated in the following lemma from \cite{BessemoulinHerdaRey2020}.

\begin{Lemma}[Equivalent norm]\label{lem.normeq.d}
 There is $\delta^*>0$ such that for all $\delta\in (0,\delta^*)$, there are positive constants $0<c_\delta<C_\delta$ such that if $\Fbf^n$ satisfies \eqref{estimLinf_f1}--\eqref{estimLinf_f2}, one has
 \[c_\delta\|\Ftbf^n\|^2\leq \Gamma^n\leq C_\delta\|\Ftbf^n\|^2.\] 
\end{Lemma}

We also have the following estimates on $(\Phibf^n_i)_{i\in\II}$. They are obtained as in the continuous setting, using discrete integration by parts and the Poincaré inequality recalled in Lemma~\ref{lem_Poincare_dis}, although the upwind fluxes give rise to additional terms that must be controlled.
\begin{Lemma}\label{lem_estim_phi_dis}
    One has for all $n\geq 0$
    \begin{align}
        & \|D_x^c\Phibf^n\|_{\Delta, x}\leq C_P\,\|\Pi^\Delta \Ftbf^n\|_\Delta, \label{estim_dxphi_dis}\\        & \| D_x^c\Phibf^{n+1}-D_x^c\Phibf^{n}\|_{\Delta, x}\leq \Delta t\left( C_1^*\|(I-\Pi^\Delta) \Ftbf^{n+1}\|_\Delta\right.\nonumber \\
        & \qquad        \left.+ C_7^*\|1-\rho_1^{n+1}\rho_2^{n+1}\|_{L^2(\T)}+ C_1^*\|\Ftbf^{n+1}\|_\Delta\right), \label{estim_dtxphi_dis}
    \end{align}
    where $C_P$ is the discrete Poincaré constant of Lemma \ref{lem_Poincare_dis}, $C_1^*$ is defined by \eqref{def.CJ*} and
    \begin{equation}\label{def.C7*}    
        C_7^*\coloneq C_P\sqrt{\rho_1^{\infty,*}+\rho_2^{\infty,*}}.
    \end{equation}      
\end{Lemma}

\begin{proof}
    Estimate \eqref{estim_dxphi_dis} can be proven exactly in the same way as \eqref{estim.dxphi}, by using the discrete Poisson equation \eqref{Poisson_d}, applying the discrete Poincaré inequality (Lemma \ref{lem_Poincare_dis}) and equality~\eqref{estim_rhot_dis}.
    
    Regarding estimate \eqref{estim_dtxphi_dis}, we obtain it by taking the difference of the scheme \eqref{Poisson_d} at times $n$ and $n+1$ and using Lemma \ref{lem_eq_moments_dis} :
    \begin{align*}
        \| D_x^c\Phibf^{n+1}-D_x^c\Phibf^{n}\|_{\Delta, x}^2&=\langle \Phibf^{n+1}-\Phibf^n,-D_x^cD_x^c(\Phibf^{n+1}-\Phibf^n)\rangle_{\Delta,x}\\
        &= \langle\Phibf^{n+1}-\Phibf^n,\rhotbf^{n+1}-\rhotbf^n\rangle_{\Delta,x}\\
        &=\Dt\left\langle \Phibf^{n+1}-\Phibf^n, -D_x^c\Jtbf^{n+1}+(1-\rho_1^{n+1}\rho_2^{n+1})\begin{pmatrix} 1\\ 1 \end{pmatrix}\right\rangle_{\Delta,x}\\
        &\quad + \Dt\left\langle \Phibf^{n+1}-\Phibf^n , \frac{\Dx}{2}(D_x^+D_x^-+D_x^-D_x^+)\Jtbf^{s,n+1}\right\rangle_{\Delta,x}.
    \end{align*}
    Using the discrete integration by parts \eqref{IPP_centre} and \eqref{IPP2_decentre}, we get
    \begin{align*}
        \| D_x^c\Phibf^{n+1}-D_x^c\Phibf^{n}\|_{\Delta, x}^2&=\Dt \langle D_x^c\Phibf^{n+1}-D_x^c\Phibf^{n},\Jtbf^{n+1}-2\Dx\,D_x^c \Jtbf^{s,n+1} \rangle_{\Delta,x}\\
        &\quad + \Dt\left\langle \Phibf^{n+1}-\Phibf^n,(1-\rho_1^{n+1}\rho_2^{n+1})\begin{pmatrix} 1\\ 1 \end{pmatrix}\right\rangle_{\Delta,x}.
    \end{align*}
    Now, thanks to the Cauchy-Schwarz inequality, estimate \eqref{estim_dxu_disc} and the discrete Poincaré inequality, we obtain
    \begin{align*}
        \| D_x^c\Phibf^{n+1}-D_x^c\Phibf^{n}\|_{\Delta, x}^2 &\leq \Dt \|D_x^c\Phibf^{n+1}-D_x^c\Phibf^{n}\|_{\Delta,x}\left[ \|\Jtbf^{n+1}\|_{\Delta,x}+2\Dx \|D_x^c\Jtbf^{s,n+1}\|_{\Delta,x}\right]\\
        & +\Dt \|\Phibf^{n+1}-\Phibf^n\|_{\Delta,x}\sqrt{\rho_1^{\infty,*}+\rho_2^{\infty,*}}\|1-\rho_1^{n+1}\rho_2^{n+1}\|_{L^2(\T)}\\
        & \leq \Dt \|D_x^c\Phibf^{n+1}-D_x^c\Phibf^{n}\|_{\Delta,x}\left( \|\Jtbf^{n+1}\|_{\Delta,x}+2\|\Jtbf^{s,n+1}\|_{\Delta,x}\right.\\ &\qquad\left.+C_7^*\|1-\rho_1^{n+1}\rho_2^{n+1}\|_{L^2(\T)}\right),
    \end{align*}
    which concludes the proof thanks to estimates \eqref{estim_J2_dis} and \eqref{estim_Jb_dis}.
\end{proof}

\begin{proof}[Proof of Theorem \ref{thm.hypoco_d}]
    We adapt the proof of Theorem \ref{thm.hypoco} to the discrete setting. For the sake of clarity, the terms are numbered as follows: discrete counterparts of the continuous terms retain the same numbering (for instance, $T_1$ becomes $T_1^\Delta$), while additional terms arising from the time and space discretizations are denoted by $R_1^\Delta$, $R_2^\Delta$, and so on.
    
    Using the microscopic coercivity \eqref{estim_coercivite_micro_dis}, we have
    \begin{align}
        \Gamma^{n+1}-\Gamma^n &\leq  -  \Delta t\,\mathcal{D}^{n+1}-C_{4}^*\|\Ftbf^{n+1}-\Ftbf^n\|_{\Delta}^2 + \delta\left(\langle \Jtbf^{n+1},D_x^c\Phibf^{n+1}\rangle_{\Delta,x}-\langle \Jtbf^{n},D_x^c\Phibf^{n}\rangle_{\Delta,x}\right)\nonumber\\
        &+\frac{\delta}{2\Delta t}\sum_{i\in\II}\Dx \left[\left((D_x^c\Phibf^{n+1})_i-(D_x^c\Phibf^{n})_i\right)^2- \left((D_x^c\Phibf^n)_i-(D_x^c\Phibf^{n-1})_i\right)^2\right].\label{hypoco_d.1}
    \end{align}
    Let us start with the term corresponding to the discrete counterpart of $\frac{\dd}{\dd t}\langle \Jtbf,\partial_x\Phibf\rangle_x$: 
    \begin{equation*}
        \lla \Jtbf^{n+1},D_x^c\Phibf^{n+1}\rra_{\Delta,x}-\lla \Jtbf^n,D_x^c\Phibf^n \rra_{\Delta,x}=\lla \Jtbf^{n+1}-\Jtbf^n,D_x^c\Phibf^{n+1}\rra_{\Delta,x}+T_4^\Delta+R_1^\Delta,
    \end{equation*}
    where 
    \begin{align*}
        T_4^\Delta&=\lla \Jtbf^{n+1},D_x^c\Phibf^{n+1}-D_x^c\Phibf^{n}\rra_{\Delta,x},\\
        R_1^\Delta&=\lla \Jtbf^n,D_x^c\Phibf^{n+1}\rra_{\Delta,x}-\lla \Jtbf^{n+1},D_x^c\Phibf^{n+1}\rra_{\Delta,x}+\lla \Jtbf^{n+1},D_x^c\Phibf^{n}\rra_{\Delta,x}-\lla \Jtbf^n,D_x^c\Phibf^{n}\rra_{\Delta,x}. 
    \end{align*}
    Using the scheme \eqref{scheme_J1_UW_d}--\eqref{scheme_J2_UW_d} for $\Jtbf$, we can rewrite
    \begin{equation*}
        \lla \Jtbf^{n+1}-\Jtbf^n,D_x^c\Phibf^{n+1}\rra_{\Delta,x}=T_1^\Delta+T_2^\Delta+T_3^\Delta+R_2^\Delta,
    \end{equation*}
    where
   \begin{align*}
        T_{1}^\Delta&= -\Delta t \lla D_x^c \Stbf^{n+1},D_x^c\Phibf^{n+1}\rra_{\Delta,x},\\
        T_{2}^\Delta&=-\Delta t\lla\begin{pmatrix} D_1^\Delta D_x^c\rhot_1^{n+1}\\ D_2^\Delta D_x^c\rhot_2^{n+1}\end{pmatrix},D_x^c\Phibf^{n+1}\rra_{\Delta,x},\\
        T_{3}^\Delta&= -\Delta t \lla \begin{pmatrix} \rhot_2^{n+1}\Jt_1^{n+1}\\ \rhot_2^{n+1}\Jt_1^{n+1} \end{pmatrix},D_x^c\Phibf^{n+1}\rra_{\Delta,x},\\
        R_2^\Delta&= \frac{\Delta t\Dx}{2}\langle (D_x^+D_x^-+D_x^-D_x^+)\overbar{\Sbf}^{n+1}_\Ftbf,D_x^c\Phibf^{n+1}\rangle_{\Delta,x}.
    \end{align*}
    Let us first deal with the residual term $R_{1}^\Delta$. Reorganizing the terms and integrating by parts, we have
    \begin{align*}
        R_{1}^\Delta&=\lla \Jtbf^n-\Jtbf^{n+1},D_x^c\Phibf^{n+1}-D_x^c\Phibf^{n}\rra_{\Delta,x}\\
        &= \lla D_x^c\Jtbf^{n+1}-D_x^c\Jtbf^n,\Phibf^{n+1}-\Phibf^n\rra_{\Delta,x}.
    \end{align*}
    The next step is to replace the discrete space derivatives of $\Jtbf$ using the scheme \eqref{scheme_rho_UW_d}, which yields $R_{1}^\Delta=R_{11}^\Delta+R_{12}^\Delta+R_{13}^\Delta$, with
    \begin{align*}
        R_{11}^\Delta&= -\frac{1}{\Delta t}\sum_{i\in\II}\Delta x\left(\rhotbf_{i}^{\,n+1}-2\rhotbf_{i}^{\,n}+\rhotbf_{i}^{\,n-1}\right)(\Phibf_i^{n+1}-\Phibf_i^n),\\
        R_{12}^\Delta&=\frac{\Dx}{2}\sum_{i\in\II}\Dx \left[\left((D_x^+D_x^-+D_x^-D_x^+)(\Jtbf^{\mathrm{s},n+1}-\Jtbf^{\mathrm{s},n})\right)_i\right](\Phibf_i^{n+1}-\Phibf_i^n),\\
        R_{13}^\Delta&= \sum_{i\in\II}\Dx \left(\rho_{1,i}^{n+1}\rho_{2,i}^{n+1}-\rho_{1,i}^n\rho_{2,i}^n\right)(\Phibf_{i}^{n+1}-\Phibf_i^n).
    \end{align*}    
    As in \cite{BCLR2025}, we combine the term $R_{11}^\Delta$ with the last line of \eqref{hypoco_d.1} using the auxiliary equation \eqref{Poisson_d}, an integration by parts, and the identity $-a(a-b)=-(a^2-b^2)/2-(a-b)^2/2$, with $a=\left(D_x^c(\Phibf^{n+1}-\Phibf_i^n)\right)_i$ and  $b=\left(D_x^c(\Phibf^{n}-\Phibf_i^{n-1})\right)_i$, therefore yielding 
    \begin{equation*}
        \frac{1}{2\Delta t}\sum_{i\in\II}\Dx \left[\left((D_x^c\Phibf^{n+1})_i-(D_x^c\Phibf^{n})_i\right)^2- \left((D_x^c\Phibf^n)_i-(D_x^c\Phibf^{n-1})_i\right)^2\right]+R_{11}^\Delta\leq 0.
    \end{equation*}
    
    Let us now turn to the estimates of $R_{12}^\Delta$ and $R_{13}^\Delta$. These terms will be bounded in terms of $\|\Ftbf^{n+1}-\Ftbf^n\|_{\Delta}^2$, which will ultimately be controlled by the additional dissipation term appearing in the microscopic coercivity estimate~\eqref{estim_coercivite_micro_dis}.
    
    First of all, we apply two integration by parts and use the auxiliary scheme \eqref{Poisson_d} to obtain
    \begin{align*}
        R_{12}^\Delta&=2\Dx \sum_{i\in\II}\Dx \left(\Jtbf_{i}^{\mathrm{s},n+1}-\Jtbf_{i}^{\mathrm{s},n}\right)\left(D_x^cD_x^c(\Phibf^{n+1}-\Phibf^n)\right)_i\\
        &=-2\Dx\sum_{i\in\II}\Dx\left(\Jtbf_{i}^{\mathrm{s},n+1}-\Jtbf_{i}^{\mathrm{s},n}\right)\left(\rhotbf_{i}^{n+1}-\rhotbf_{i}^n\right).
    \end{align*}
    Using the moments estimates \eqref{estim_rhot_dis} and \eqref{estim_Jb_dis}, and assuming without loss of generality that $\Dx<1$, we finally get
    \begin{equation}\label{estim_T132}
        |R_{12}^\Delta|\leq 2\Dx \,C_1^*\,\|\Ftbf^{n+1}-\Ftbf^n\|_\Delta^2\leq 2\,C_1^*\,\|\Ftbf^{n+1}-\Ftbf^n\|_\Delta^2.
    \end{equation}
     Using the Cauchy-Schwarz inequality, we obtain
    \begin{equation}\label{T133_1}
        |R_{13}^\Delta|\leq \|\rho_1^{n+1}\rho_2^{n+1}-\rho_1^n\rho_2^n\|_{L^2(\T)}\|\Phibf^{n+1}-\Phibf^n\|_{\Delta,x}.
    \end{equation}
     We treat the first factor by linearizing and using the $L^\infty$ bounds:
    \begin{align*}
        \|\rho_1^{n+1}\rho_2^{n+1}-\rho_1^n\rho_2^n\|_{L^2(\T)}^2 &=  \|\rho_1^{n+1}(\rho_2^{n+1}-\rho_2^n)+\rho_2^n(\rho_1^{n+1}-\rho_1^n)\|_{L^2(\T)}^2\\
        &\leq 2 \left(\rho_M^2\|\rho_2^{n+1}-\rho_2^n\|^2_{L^2(\T)}+\rho_m^{-2}\|\rho_1^{n+1}-\rho_1^n\|_{L^2(\T)}^2\right)\\
        & \leq 2 \max(\rho_M^2\rho_2^{\infty,*},\rho_m^{-2}\rho_{1}^{\infty,*})\|\rhobf_{\Fbf}^{n+1}-\rhobf_\Fbf^n\|_{\Delta,x}^2,
    \end{align*}
    from which we deduce
    \[
    \|\rho_1^{n+1}\rho_2^{n+1}-\rho_1^n\rho_2^n\|_{L^2(\T)} \leq C_{8}^*\,\|\Ftbf^{n+1}-\Ftbf^n\|_{\Delta},
    \]
    with $C_{8}^*=2 \max(\rho_M^2\rho_2^{\infty,*},\rho_m^{-2}\rho_{1}^{\infty,*})$. Now, for the second factor of \eqref{T133_1}, remark that by the Poincaré inequality and an integration by parts, 
    \begin{align*}
        \|\Phibf^{n+1}-\Phibf^n\|_{\Delta,x}^2&\leq C_P^2\,\|D_x^c(\Phibf^{n+1}-\Phibf^n)\|_{\Delta,x}^2\\
        &\leq -C_P^2\lla \Phibf^{n+1}-\Phibf^n,D_x^cD_x^c(\Phibf^{n+1}-\Phibf^n)\rra_{\Delta,x}.
    \end{align*}
    Using now the auxiliary scheme \eqref{Poisson_d} and the Cauchy-Schwarz inequality, we get
    \begin{align*}
        \|\Phibf^{n+1}-\Phibf^n\|_{\Delta,x}^2&\leq C_P^2\lla \Phibf^{n+1}-\Phibf^n,\rhotbf^{\,n+1}-\rhotbf^{\,n}\rra_{\Delta,x}\\
        &\leq C_P^2 \,\|\Phibf^{n+1}-\Phibf^n\|_{\Delta,x}\|\rhotbf^{n+1}-\rhotbf^{\,n}\|_{\Delta,x},
    \end{align*}
    from which we obtain
    \[
    \|\Phibf^{n+1}-\Phibf^n\|_{\Delta,x}\leq C_P^2 \|\rhotbf^{\,n+1}-\rhotbf^{\,n}\|_{\Delta,x}\leq C_P^2 \,\|\Ftbf^{n+1}-\Ftbf^n\|_{\Delta}.
    \]
    Finally, we get
    \begin{equation}\label{estim_T133}
        |R_{13}^\Delta|\leq C_{8}^*\,C_P^2\,\|\Ftbf^{n+1}-\Ftbf^n\|_{\Delta}^2.
    \end{equation}
    
   The rest of the proof follows the same steps as in the continuous setting, with the additional task of controlling the numerical viscosity. The terms $T_k^\Delta$, for $k=1,2,3$ correspond respectively to the discrete counterparts of the continuous terms $T_k$, and are handled in exactly the same manner, yielding
    \begin{align*}
        T_{1}^\Delta&\leq \Delta t \,C_2^* \|(I-\Pi^\Delta)\Ftbf^{n+1}\|_{\Delta}\|\Pi^\Delta\Ftbf^{n+1}\|_\Delta,\\
        T_{2}^\Delta&\leq -\Delta t\,\min(\underline{D}_1,\underline{D}_2)\|\Pi^\Delta\Ftbf^{n+1}\|_\Delta^2,\\
        T_{3}^\Delta&\leq \Delta t \, C_9^*\|(I-\Pi^\Delta)\Ftbf^{n+1}\|_{\Delta}\|\Pi^\Delta\Ftbf^{n+1}\|_\Delta,
    \end{align*}    
    with $ C_9^*=\max(\rho_m^{-1},\rho_M)\,C_1^*C_P$.
    
    The term $R_{2}^\Delta$ arises from the numerical viscosity introduced by the upwind fluxes. It can be estimated by combining \eqref{estim_dxu_disc} with the discrete moment estimates \eqref{estim_rhot_dis} and \eqref{estim_Sb_dis}, which yields
    \[R_{2}^\Delta \leq 2\,\Delta t\, C^*_{3}\,\|(I-\Pi^\Delta)\Ftbf^{n+1}\|_{\Delta}\|\Pi^\Delta\Ftbf^{n+1}\|_\Delta.\]
    Gathering all these estimates, we conclude that
    \begin{align*}
        \lla \Jtbf^{n+1}-\Jtbf^n,D_x^c\Phibf^{n+1}\rra_{\Delta,x}&\leq -\Delta t\,\min(\underline{D}_1,\underline{D}_2)\|\Pi^\Delta\Ftbf^{n+1}\|_\Delta^2\\
        &+\Delta t\left(C_2^*+C_9^*+2C_3^*\right)\,\|(I-\Pi^\Delta)\Ftbf^{n+1}\|_{\Delta}\|\Pi^\Delta\Ftbf^{n+1}\|_\Delta.
    \end{align*}
    
    Finally, the term $T_{4}^\Delta$ is the discrete counterpart of the term $T_4$. Using the Cauchy-Schwarz inequality and estimates \eqref{estim_J2_dis} and \eqref{estim_dtxphi_dis}, we have
    \begin{gather*}
        T_{4}^\Delta\leq \Delta t(C_1^*)^2\|(I-\Pi^\Delta)\Ftbf^{n+1}\|_\Delta^2+\Delta t C_1^*C_7^*\|(I-\Pi^\Delta)\Ftbf^{n+1}\|_\Delta\|1-\rho_1^{n+1}\rho_2^{n+1}\|_{L^2(\T)}\\
        + \Delta t (C_1^*)^2\|(I-\Pi^\Delta)\Ftbf^{n+1}\|_{\Delta}\|\Ftbf^{n+1}\|_\Delta.
    \end{gather*}
    
    Then, using Young's inequality, we have for $\eta>0$ (which will be chosen conveniently later) some constants $C_\eta>0$ and $C_{10}^*$ such that
    \[
    T_{4}^\Delta\leq \Delta t C_\eta\,\|(I-\Pi^\Delta)\Ftbf^{n+1}\|_\Delta^2+\Delta t C_{10}^*\|1-\rho_1^{n+1}\rho_2^{n+1}\|_{L^2(\T)}^2+\Delta t\frac{\eta}{2}\,\|\Ftbf^{n+1}\|_\Delta^2.
    \]
    By collecting the estimates of all the terms, we obtain
    \begin{align*}
        \Gamma^{n+1}-\Gamma^n\leq &-\Delta t\,\mathcal{D}^{n+1}-C^*_{4}\|\Ftbf^{n+1}-\Ftbf^n\|_\Delta^2-\Delta t\,\delta\min (\underbar{D}_1,\underbar{D}_2)\|\Pi^\Delta\Ftbf^{n+1}\|^2_\Delta\\
        & +\Delta t\,\delta(C_2^*+C_9^*+2C_{3}^*)\|(I-\Pi^\Delta)\Ftbf^{n+1}\|_\Delta\|\Pi^\Delta\Ftbf^{n+1}\|_\Delta\\
        & + \Delta t\,\delta\,C_\eta\|(I-\Pi^\Delta)\Ftbf^{n+1}\|_\Delta^2+\Delta t\,\delta\,C_{10}^*\|1-\rho_1^{n+1}\rho_2^{n+1}\|_{L^2(\T)}^2\\
        &+\delta\frac{\eta}{2}\,\|\Ftbf^{n+1}\|_\Delta^2+2\delta\, C_1^*\|\Ftbf^{n+1}-\Ftbf^n\|_\Delta^2+\delta\,C_{8}^*C_P^2\|\Ftbf^{n+1}-\Ftbf^n\|_\Delta^2.
    \end{align*}
    Now, using that 
    \[
    \|\Ftbf^{n+1}\|^2_\Delta=\|\Pi^\Delta\Ftbf^{n+1}\|_\Delta^2+\|(I-\Pi^\Delta)\Ftbf^{n+1}\|_\Delta^2,
    \]
    and Lemma \ref{lem_estimD_d},  we finally get
    \begin{align*}
        \Gamma^{n+1}-\Gamma^n\leq & -\left(C^*_{4}-\delta(2 C_1^*+C_{8}^*C_P^2)\right)\|\Ftbf^{n+1}-\Ftbf^n\|_\Delta^2\\
        & -\delta\Delta t\left(\min(\underbar{D}_1,\underbar{D}_2)-\frac{\eta}{2}\right)\|\Pi^\Delta\Ftbf^{n+1}\|_\Delta^2\\
        &+\delta\Delta t\left(C_2^*+C_9^*+2C_{3}^*\right)\|(I-\Pi^\Delta)\Ftbf^{n+1}\|_\Delta\|\Pi^\Delta\Ftbf^{n+1}\|_\Delta\\
        &-\Delta t\left(C_6^*-\delta\left(C_\eta+\frac{\eta}{2}\right)\right)\|(I-\Pi^\Delta)\Ftbf^{n+1}\|_\Delta^2\\
        &-\Delta t\left(C_5^*-\delta C_{10}^*\right)\|1-\rho_1^{n+1}\rho_2^{n+1}\|_{L^2(\T)}^2.
    \end{align*}
     First, we fix $\eta >0$ in such a way that the second term of the right hand side is nonpositive, namely $\eta \leq 2 \min(\underbar{D}_1,\underbar{D}_2)$. Second, we also assume that $0<\delta<\delta_0^*$, with
    \[
    \delta_0^*=\min\left(\frac{C_{4}^*}{2C_1^*+C_{8}^*C_P^2},\frac{C_5^*}{C_{10}^*},\frac{C_{9}^*}{C_\eta+\eta/2}\right), 
    \]
    ensuring that the first and the last two terms are nonpositive.
    
    With such an $\eta$ and $\delta$, we now have
    \begin{gather*}
        \frac{\Gamma^{n+1}-\Gamma^n}{\Delta t}\leq  -\delta\alpha_1\|\Pi^\Delta\Ftbf^{n+1}\|_\Delta^2 - \left(C_6^*-\delta \alpha_2\right)\|(I-\Pi^\Delta)\Ftbf^{n+1}\|_\Delta^2\\
        +\delta\alpha_3 \|(I-\Pi^\Delta)\Ftbf^{n+1}\|_\Delta\|\Pi^\Delta\Ftbf^{n+1}\|_\Delta,
    \end{gather*} 
    with $\alpha_1=\min(\underbar{D}_1,\underbar{D}_2)-\eta/2$, $\alpha_2=C_\eta+\eta/2$, and $\alpha_3=C_2^*+C_9^*+2C_{3}^*$. If we choose $\delta<\delta_1$ with
    \[
    \delta_1<\frac{C_6^*\alpha_1}{\alpha_3^2+\alpha_1\alpha_2},
    \]
    we then get
    \[
    \frac{\Gamma^{n+1}-\Gamma^n}{\Delta t}\leq- \frac{\delta\alpha_1}{2}\|\Pi^\Delta\Ftbf^{n+1}\|_\Delta^{2}-\frac{1}{2}(C_6^*-\delta\alpha_2)\|(I-\Pi^\Delta)\Ftbf^{n+1}\|_\Delta^2,
    \]
    from which we conclude the proof using Lemma~\ref{lem.normeq.d}.
\end{proof}

\section*{Acknowledgments}
TR received funding from the European Union's Horizon Europe research and innovation program under the Marie Skłodowska-Curie Doctoral Network DataHyking (Grant No. 101072546), and from UniCA$_{JEDI}$ Investments in the Future project managed by the National Research Agency (ANR) with the reference number ANR-15-IDEX-01.	 

TL received funding from the European Research Council (ERC) under the European Union’s Horizon 2020 research and innovation program (grant agreement No 865711) and from the European Union's Horizon Europe research and innovation program under the Marie Skłodowska-Curie Doctoral Network DataHyking (Grant No. 101072546).

MBC  benefits from the support of the French government “Investissements d’Avenir” program integrated to France 2030, under the reference ANR-11-LABX-0020-01, and ANR Project Cookie (ANR-25-CE40-5565-02).

TR would also like to thanks Isabelle Tristani for fruitful discussions about the continuous part of this work. 

\bibliographystyle{acm}
\bibliography{biblio_hypoco_system}

\appendix

\section{Proof of the existence and maximum principle}\label{app:ExistenceMaxPrinciple}

To prove Theorem~\ref{thm_bounds_existence}, we follow along the lines of  \cite[Theorem 3]{BCLR2025}, adapted to the upwind fluxes. Let us first introduce the following truncated version of the scheme \eqref{scheme_1_nonlin_d1}--\eqref{scheme_2_nonlin_d1}:
\begin{align}
    & \frac{f_{1,ij}^{n+1}-f_{1,ij}^n}{\Delta t}+\frac{1}{\Delta x\,\Delta v}\left(\mathcal{F}_{\iph,}^{n+1}-\mathcal{F}_{\imh,j}^{n+1}\right)=\chi_{1,j}-\overbar{\rho}_{2,i}^{n+1}\overbar{f}_{1,ij}^{n+1} ,\label{scheme_1_nonlin_trunc}\\
    & \frac{f_{2,ij}^{n+1}-f_{2,ij}^n}{\Delta t}+\frac{1}{\Delta x\,\Delta v}\left(\mathcal{G}_{\iph,j}^{n+1}-\mathcal{G}_{\imh,j}^{n+1}\right)=\chi_{2,j}-\overbar{\rho}_{1,i}^{n+1}\overbar{f}_{2,ij}^{n+1},\label{scheme_2_nonlin_trunc}
\end{align}
with the upwind fluxes \eqref{flux_F_UW_sd1}--\eqref{flux_G_UW_sd1}, and where the truncated quantities are defined by
\begin{equation}\label{def_f_trunc}
    \overbar{f}_{1,ij}\coloneq\left\{\begin{array}{ll}
        \rho_m\chi_{1,j} & \text{ if } f_{1,ij}\leq \rho_m\chi_{1,j},  \\
        f_{1,ij} & \text{ if } \rho_m\chi_{1,j}\leq f_{1,ij}\leq \rho_M\chi_{1,j},\\
        \rho_M\chi_{1,j} & \text{ if } f_{1,ij} \geq \rho_M\chi_{1,j},
    \end{array}\right.
\end{equation}
and
\begin{equation}\label{def_g_trunc}
    \overbar{f}_{2,ij}\coloneq\left\{\begin{array}{ll}
        \rho_M\inv\chi_{2,j} & \text{ if } f_{2,ij}\leq \rho_M\inv\chi_{2,j},  \\
        f_{2,ij} & \text{ if } \rho_M\inv\chi_{2,j}\leq f_{2,ij}\leq \rho_m\inv\chi_{2,j},\\
        \rho_m\inv\chi_{2,j} & \text{ if } f_{2,ij} \geq \rho_m\inv\chi_{2,j}.
    \end{array}\right.
\end{equation}
The corresponding truncated densities are then given by
\begin{equation*}
    \overbar{\rho}_{1,i}=\sum_{j\in\J}\Delta v\,\overbar{f}_{1,ij},\qquad \overbar{\rho}_{2,i}=\sum_{j\in\J}\Delta v\,\overbar{f}_{2,ij} \qquad \forall i\in\II.
\end{equation*}

The proof of the existence of solutions to \eqref{scheme_1_nonlin_trunc}--\eqref{scheme_2_nonlin_trunc} relies on the following result which is proven for example in \cite[Lemma 1.4, Chapter II]{temam1984}.
\begin{Lemma}\label{lem_existence_Navier}
    Let $X$ be a finite dimensional Hilbert space with scalar product $\langle\cdot,\cdot\rangle_X$ and associated norm $\|\cdot\|_X$. Let $P:X\rightarrow X$ be a continuous mapping such that $\langle P(\xi),\xi\rangle_X >0$ for all $\xi\in X$ such that $\|\xi\|_X=k$ for some fixed $k>0$. Then there exists $\xi_0\in X$ such that $\|\xi_0\|_X\leq k$ and $P(\xi)=0$.
\end{Lemma}
We apply this result with $X=\left(\R^{(2L+1)N} \right)^2$ where for $\Fbf=(f_1,f_2)^\top \in X$ and $\Gbf=(g_1,g_2)^\top\in X $, we set
\begin{equation*}
    \langle \Fbf,\Gbf\rangle_X=\langle f_1,g_1\rangle+\langle f_2,g_2\rangle,
\end{equation*}
$\langle \cdot,\cdot \rangle$ being the classical Euclidian dot product on $\R^{(2L+1)N}$.
Then, let us define
\begin{equation*}
    P:\begin{pmatrix} f_1\\f_2 \end{pmatrix}\mapsto
    \begin{pmatrix} \Delta x\Delta v(f_{1,ij}-f_{1,ij}^n)+\Delta t(\mathcal{F}_{\iph,j}-\mathcal{F}_{\imh,j})-\Delta x\Delta v\Delta t(\chi_{1,j}-\overbar{\rho}_{2,i}\overbar{f}_{1,ij})\\ 
    \Delta x\Delta v(f_{2,ij}-f_{2,ij}^n)+\Delta t(\mathcal{G}_{\iph,j}-\mathcal{G}_{\imh,j})-\Delta x\Delta v\Delta t(\chi_{1,j}-\overbar{\rho}_{1,i}\overbar{f}_{2,ij})
    \end{pmatrix},
\end{equation*}
where $\mathcal{F}$ and $\mathcal{G}$ are the numerical fluxes \eqref{flux_1_UW_d}--\eqref{flux_2_UW_d} computed with $f_{1,ij}$ and $f_{2,ij}$. We shall show that
\begin{equation*}
    \lla P(\Fbf), \Fbf \rra_X > 0, \quad \text{ for all } \Fbf=( f_1,f_2)^\top  \text{ such that }\left\|\Fbf\right\|_X=k,
\end{equation*}
where $k$ is taken large enough. Indeed, this scalar product splits into 3 terms, 
\begin{equation*}
    \lla P(\Fbf), \Fbf \rra_X=A_1+A_2+A_3,
\end{equation*} 
where
\begin{align*}
    A_1 &= \sum_{(i,j)\in\II\times\J} \Delta x\Delta v\left((f_{1,ij}-f_{1,ij}^n)f_{1,ij} + (f_{2,ij}-f_{2,ij}^n)f_{2,ij}\right),\\
    A_2 &= \Delta t\sum_{(i,j)\in\II\times\J} \left((\mathcal{F}_{\iph,j}-\mathcal{F}_{\imh,j})f_{1,ij} + (\mathcal{G}_{\iph,j}-\mathcal{G}_{\imh,j})f_{2,ij}\right),\\
    A_3 &= -\Delta t\sum_{(i,j)\in\II\times\J} \Delta x\Delta v\left((\chi_{1,j}-\overbar{\rho}_{2,i}\overbar{f_1}_{ij})f_{1,ij} + (\chi_{2,j}-\overbar{\rho}_{1,i}\overbar{f_2}_{ij})f_{2,ij}\right).
\end{align*}
Using the relation $a(a-b)\geq(a^2-b^2)/2$ one gets
\begin{equation}\label{estim_A1}
    A_1 \geq\frac{1}{2}\left( \left\| \Fbf\right\|_X^2 - \left\| \Fbf^n\right\|_X^2\right).
\end{equation}
Focusing on the first species, by definition of the numerical fluxes, one can use similar computations as in the proof of Lemma~\ref{lem_disc_microcoercivity}. In particular, using \eqref{flux_F_UW_sd1}, a discrete integration by parts and the periodic boundary conditions, one obtains
\begin{equation}\label{estim_A2}
    \begin{aligned}
      \sum_{i,j}(\mathcal{F}_{\iph,j}-\mathcal{F}_{\imh,j})f_{1,ij} &= -\sum_{i,j}\mathcal{F}_{\iph,j}(f_{1,i+1,j}-f_{1,ij})\\
        &= -\sum_{i\in\II, j>0} \Dv\, v_j f_{1,ij}\left(f_{1,i+1,j}-f_{1,i,j}\right)\\
        &\qquad- \sum_{i\in\II, j<0} \Dv\,v_jf_{1,i+1,j}(f_{1,i+1,j}-f_{1,ij})\\
        &\geq \frac{1}{2}\sum_{i,j}\Dv\, v_j\left((f_{1,ij})^2-(f_{1,i+1,j})^2\right)=0.
    \end{aligned}
\end{equation}

Now, by assumptions \eqref{hyp_chi_dis}, the velocity profiles are bounded:
\begin{equation*}
    \exists C_\chi>0 \quad\text{ such that }\quad \forall j\in\J, \quad 0\leq \chi_{1,j},\chi_{2,j}\leq C_\chi.
\end{equation*}
Then, by definition of the truncated quantities, there exists a constant $C>0$ such that
\begin{equation*}
    |\chi_{1,j}-\overbar{\rho}_{2,i}\overbar{f}_{1,ij}|\leq C\quad\text{and}\quad |\chi_{2,j}-\overbar{\rho}_{1,i}\overbar{f}_{2,ij}|\leq C.
\end{equation*}
Applying the Cauchy-Schwarz inequality on $A_3$, one gets
\begin{equation}\label{estim_A3}
    A_3\geq -\Delta t C \left\| \Fbf\right\|_X.
\end{equation}
Combining \eqref{estim_A1}, \eqref{estim_A2} and \eqref{estim_A3} yields the estimate
\begin{equation*}
    \lla P(\Fbf), \Fbf \rra_X\geq 
    \frac{1}{2}\left( \left\| \Fbf\right\|_X^2 - \left\| \Fbf^n\right\|_X^2\right) -\Delta t C \left\| \Fbf\right\|_X.
\end{equation*} 
The right hand side is a second order polynomial in $\left\| \Fbf\right\|_X$ with a positive leading coefficient. Since $\left\| \Fbf^n\right\|_X$ is a constant in this context, there exists $k>0$ such that if $\left\| \Fbf\right\|_X\geq k$ then 
\begin{equation*}
    \lla P(\Fbf), \Fbf \rra_X> 0.
\end{equation*} 
Finally, applying Lemma \ref{lem_existence_Navier}, one obtains existence of $\Fbf^{n+1}$ such that $P(\Fbf^{n+1})=0$, therefore ensuring existence of a solution to the truncated scheme \eqref{scheme_1_nonlin_trunc}-\eqref{scheme_2_nonlin_trunc}.

\begin{Lemma}\label{lem_bornes_trunc}
    If  $(f_{1,ij}^n,f_{2,ij}^n)^\top_{i\in\II,j\in\J}$ satisfies the estimates \eqref{estimLinf_f1_d}--\eqref{estimLinf_f2_d}, then any solution\\ $(f_{1,ij}^{n+1},f_{2,ij}^{n+1})^\top_{i\in\II,j\in\J}$ to the truncated scheme \eqref{scheme_1_nonlin_trunc}--\eqref{scheme_2_nonlin_trunc} also satisfies these estimates.
\end{Lemma}

\begin{proof}
    Let us focus on proving that any solution to the nonlinear truncated scheme \eqref{scheme_1_nonlin_trunc}– \eqref{scheme_2_nonlin_trunc} satisfies
    \begin{equation*}
        f_{1,ij}^{n+1} \geq\rho_m\chi_{1,j}.
    \end{equation*}
    We start by setting $(i,j)\in\II\times\J$ such that 
    \begin{equation}\label{eq_ij_bounds}
        f_{1,ij}^{n+1} - \rho_m\chi_{1,j} = \underset{(k,l)\in\II\times\J}{\min}(f_{1,kl}^{n+1} - \rho_m\chi_{1,l}).
    \end{equation}
    Our aim is now to show that this quantity is nonnegative. To this end, we multiply the equation in \eqref{scheme_1_nonlin_trunc} corresponding to the fixed pair $(i,j)$ by
    \begin{equation}\label{eq_negPart}
        (f_{1,ij}^{n+1} - \rho_m\chi_{1,j})^-=\min(0,f_{1,ij}^{n+1} - \rho_m\chi_{1,j})\leq 0.
    \end{equation}
    It yields an expression of the form $B_1=B_2+B_3$ where
    \begin{align*}
        B_1 &= \Delta x\Delta v(f_{1,ij}^{n+1}-f_{1,ij}^{n})(f_{1,ij}^{n+1} - \rho_m\chi_{1,j})^-,\\
        B_2 &= -\Delta t(\mathcal{F}_{\iph,j}^{n+1}-\mathcal{F}_{\imh,j}^{n+1})(f_{1,ij}^{n+1} - \rho_m\chi_{1,j})^-,\\
        B_3 &=  \Delta t\Delta x\Delta v(\chi_{1,j}-\overbar{\rho}_{2,i}^{n+1}\overbar{f}_{1,ij}^{n+1})(f_{1,ij}^{n+1} - \rho_m\chi_{1,j})^-.
    \end{align*}
    Starting with $B_1$, we add and subtract $\rho_m\chi_{1,j}$ to obtain
    \begin{equation*}
        B_1 = \Delta x\Delta v( (f_{1,ij}^{n+1}-\rho_m\chi_{1,j}) - (f_{1,ij}^{n}-\rho_m\chi_{1,j}))(f_{1,ij}^{n+1} - \rho_m\chi_{1,j})^-
    \end{equation*}
    Then, under the hypothesis that $f^n_{1,ij}$ satisfies the bounds \eqref{estimLinf_f1_d}, and using the definition \eqref{eq_negPart} of the negative part, 
    \begin{equation}\label{estim_B1}
        B_1 \geq \Delta x\Delta v\,(f_{1,ij}^{n+1}-\rho_m\chi_{1,j})(f_{1,ij}^{n+1} - \rho_m\chi_{1,j})^- \geq 0.
    \end{equation}
    Next, we turn our attention to $B_3$. We need to consider two cases:
    \begin{itemize}
        \item If $f_{1,ij}^{n+1}\geq\rho_m\chi_{1,j}$ then $B_3=0$.
        \item Else, by definition of the truncated quantities \eqref{def_f_trunc} and \eqref{def_g_trunc}, one has $\overbar{f}_{1,ij}^{n+1} = \rho_m\chi_{1,j}$ and since the discrete velocity profiles have unit mass, $\overbar{\rho}_{2,i}^{n+1}\leq\rho_m\inv$. Therefore, 
        \begin{equation*}
            \chi_{1,j}-\overbar{\rho}_{2,i}^{n+1}\overbar{f}_{1,ij}^{n+1} \geq \chi_{1,j}-\frac{\rho_m}{\rho_m}\chi_{1,j}= 0.
        \end{equation*}
    \end{itemize}
    As a result, one gets
    \begin{equation}\label{estim_B3}
        B_3\leq 0.
    \end{equation}
    Finally, the sign of $B_2$ relies on the monotonicity of the numerical flux. More precisely, let us consider a monotone numerical flux in a general two-point approximation form: 
    \begin{equation*}
        \mathcal{F}_{\iph,j}^{n+1}=\varphi_j(f_{1,ij}^{n+1},f_{1,i+1,j}^{n+1}),
    \end{equation*}
    where $a\mapsto \varphi_j(a,\cdot)$ is assumed to be nondecreasing and $b\mapsto \varphi_j(\cdot,b)$ is assumed to be nonincreasing. Then, the balance of fluxes at cell $\mathcal{X}_i$ rewrites as
    \begin{align}\label{eq_fluxmono_varphi}
        \mathcal{F}_{\iph,j}^{n+1}-\mathcal{F}_{\imh,j}^{n+1}=\, &\varphi_j(f_{1,ij}^{n+1},f_{1,i+1,j}^{n+1}) - \varphi_j(f_{1,ij}^{n+1},f_{1,i,j}^{n+1})\nonumber\\
        & + \varphi_j(f_{1,ij}^{n+1},f_{1,ij}^{n+1}) - \varphi_j(f_{1,i-1,j}^{n+1},f_{1,ij}^{n+1}).
    \end{align}
    From our choice \eqref{eq_ij_bounds} of the pair $(i,j)$, we deduce that
    \begin{equation*}
        f_{1,ij}^{n+1} - \rho_m\chi_{1,j} \leq f_{1,kl}^{n+1} - \rho_m\chi_{1,l},\quad\forall (k,l)\in\II\times\J,
    \end{equation*}
    therefore in particular, for every $k\in\II$, $f_{1,ij}^{n+1}\leq f_{1,kj}^{n+1}$. Consequently, due to the monotonicity of the function $\varphi_j$, it yields
    \begin{equation}\label{estim_B2}
        B_2 \leq 0.
    \end{equation}

    Gathering \eqref{estim_B1}, \eqref{estim_B3} and \eqref{estim_B2} into $B_1=B_2+B_3$, and using the monotonicity of the upwind fluxes, we get that $B_1\leq0$, and consequently that
    \begin{equation*}
        f_{1,i,j}^{n+1} - \rho_m\chi_{1,j} \geq 0.
    \end{equation*}
    Lastly, since we chose $(i,j)$ satisfying \eqref{eq_ij_bounds}, 
    \begin{equation*}
        f_{1,i,j}^{n+1} \geq \rho_m\chi_{1,j},\quad \forall (i,j)\in\II\times\J.
    \end{equation*}
    The remaining bounds can then be obtained following the same steps.
\end{proof}

\begin{proof}[Proof of Theorem \ref{thm_bounds_existence}]
    We proceed by induction on $n$. The case $n=0$ is satisfied by assumption. Suppose now that there exists $(f_1^n,f_2^n)^\top$ satisfying the bounds \eqref{estimLinf_f1_d} and \eqref{estimLinf_f2_d}. Then, we showed the existence of $(f_1^{n+1},f_2^{n+1})^\top$ solution to the truncated scheme \eqref{scheme_1_nonlin_trunc}-\eqref{scheme_2_nonlin_trunc}, which satisfies \eqref{estimLinf_f1_d} and \eqref{estimLinf_f2_d} according to Lemma \ref{lem_bornes_trunc}. But then, we have
    \begin{align*}
        &\overbar{f_1}^{n+1} = f_1^{n+1},\quad \overbar{f_2}^{n+1} = f_2^{n+1},\\
        &\overbar{\rho_1}^{n+1} = \rho_1^{n+1},\quad \overbar{\rho_2}^{n+1} = \rho_2^{n+1},
    \end{align*}
    meaning that $(f_1^{n+1},f_2^{n+1})^\top$ is indeed a solution to the original nonlinear scheme \eqref{scheme_1_nonlin_d1}-\eqref{scheme_1_nonlin_d1}, which concludes the proof.
\end{proof}

\end{document}